\documentclass[reqno]{amsart}
\usepackage{amsmath,amsthm,amscd,amssymb,amsfonts, amsbsy}
\usepackage[usenames, dvipsnames]{color}
\usepackage{pxfonts}
\usepackage{enumerate}
\usepackage{latexsym}
\usepackage{mathrsfs}
\usepackage{xparse}

\allowdisplaybreaks[4]

\NewDocumentCommand{\oldnorm}{sO{}m}{%
  {\IfBooleanTF{#1}
    {\oldnormaux{\left|}{\right|}{#3}}
    {\oldnormaux{#2|}{#2|}{#3}}}
}
\makeatletter
\newcommand{\oldnormaux}[3]{\mathpalette\oldnormaux@i{{#1}{#2}{#3}}}
\newcommand{\oldnormaux@i}[2]{\oldnormaux@ii#1#2}
\newcommand{\oldnormaux@ii}[4]{%
  \sbox\z@{$\m@th#1#2#4#3$}%
  \sbox\tw@{$\m@th\|$}%
  \mathopen{\hbox to\wd\tw@{\hss\vrule height \ht\z@ depth \dp\z@ width .3\wd\tw@\hss}}%
  #4
  \mathclose{\hbox to\wd\tw@{\hss\vrule height \ht\z@ depth \dp\z@ width .3\wd\tw@\hss}}%
}
\makeatother

\theoremstyle{plain}
\newtheorem{theorem}[equation]{Theorem}
\newtheorem{lemma}[equation]{Lemma}
\newtheorem{corollary}[equation]{Corollary}

\theoremstyle{definition}

\theoremstyle{remark}
\newtheorem{remark}[equation]{Remark}

\DeclareMathOperator{\supp}{supp}

\DeclareMathOperator{\Div}{div}
\DeclareMathOperator*{\esssup}{ess\,sup}

\providecommand{\set}[1]{\{#1\}}

\providecommand{\abs}[1]{\lvert#1\rvert}
\providecommand{\Abs}[1]{\left\lvert#1\right\rvert}

\providecommand{\norm}[1]{\lVert#1\rVert}

\newcommand{\bR}{\mathbb{R}}

\newcommand*{\tran}{^{\mkern-1.5mu\mathsf{T}}}

\renewcommand{\vec}[1]{\boldsymbol{#1}}

\numberwithin{equation}{section}

\begin{document}


\title[Green's function for parabolic equations with lower order terms]{Green's function for second order parabolic equations with singular lower order coefficients}

\author[S. Kim]{Seick Kim}
\address[S. Kim]{Department of Mathematics, Yonsei University, 50 Yonsei-Ro, Seodaemun-gu, Seoul 03722, Republic of Korea.}
\email{kimseick@yonsei.ac.kr}
\thanks{S. Kim was partially supported by the National Research Foundation of Korea under agreements NRF-2019R1A2C2002724 and NRF-20151009350.}

\author[L. Xu]{Longjuan Xu}
\address[L. Xu]{Department of Mathematics, Yonsei University, 50 Yonsei-Ro, Seodaemun-gu, Seoul 03722, Republic of Korea.}
\email{ljxu311@163.com}

\begin{abstract}
We construct Green's functions for second order parabolic operators of the form $Pu=\partial_t u-{\rm div}({\bf A} \nabla u+ \boldsymbol{b}u)+ \boldsymbol{c} \cdot \nabla u+du$ in $(-\infty, \infty) \times \Omega$, where $\Omega$ is an open connected set in $\mathbb{R}^n$.
It is not necessary that $\Omega$ to be bounded and $\Omega = \mathbb{R}^n$ is not excluded.
We assume that the leading coefficients $\bf A$ are bounded and measurable and the lower order coefficients $\boldsymbol{b}$, $\boldsymbol{c}$, and $d$ belong to critical mixed norm Lebesgue spaces and satisfy the conditions $d-{\rm div} \boldsymbol{b} \ge 0$ and  ${\rm div}(\boldsymbol{b}-\boldsymbol{c}) \ge 0$.
We show that the Green's function has the Gaussian bound in the entire $(-\infty, \infty) \times \Omega$.
\end{abstract}

\maketitle

\section{Introduction}

In this paper, we are concerned with Green's functions of the second order parabolic equations of divergence form
\begin{align}		\label{main para eq}
Pu&=\partial_t u-\sum_{i,j=1}^{n}D_i (a^{ij}(t,x)D_j  u+b^i (t,x)u)+\sum_{i=1}^{n}c^i (t,x)D_i u+d(t,x)u\nonumber\\
&=\partial_t u-\Div({\bf A}\nabla u+\vec bu)+\vec c\cdot\nabla u+du
\end{align}
in a cylindrical domain $\mathscr D= (a,b) \times \Omega \subset \bR^{n+1}$, where $-\infty \le a <b\le +\infty$ and $\Omega$ is an open connected set in $\bR^n$ with $n \ge 1$.
It is not necessary that $\Omega$ to be bounded and $\Omega = \bR^n$ is not excluded.
In the case when $\Omega=\bR^n$, the Green's function is usually called the fundamental solution.

By Green's function for the operator $P$, we mean a function $G(t,x,s,y)$ which satisfies the following:
\begin{align*}
P G(\cdot,\cdot,s,y)=0 &\quad\text{ in }\; (s,\infty) \times \Omega, \\
G(t,x,s,y)=\delta_y(x) &\quad\text{ on }\; \set{t=s} \times \Omega,
\end{align*}
where $\delta_y(\cdot)$ is a Dirac delta function. See Theorem~\ref{thm_main} for the precise definition.

Before describing the remaining assumption on $P$, we introduce the function space $L_{p,q}(\mathscr D)$, the usual Lebesgue space with mixed norm.
Let $t$ denote points on the real line $\bR$ and $x=(x_1, \ldots, x_n)$ denote points in the $n$-dimensional Euclidean space $\bR^n$.
For $f \in L_{p,q}(\mathscr D)$ with $1 \le p$, $q<\infty$, we define
\[
\norm{f}_{p,q}= \norm{f}_{L_{p,q}(\mathscr D)}:= \left(\int_a^b \left(\int_{\Omega} \abs{f(t,x)}^{p}\ dx\right)^{q/p}\ dt\right)^{1/q}.
\]
In case either $p$ or $q$ is infinite, $\norm{f}_{p,q}$ is defined in a similar fashion using essential supremum rather than integrals.
We denote $L_{p,p}(\mathscr D)$ by $L_p(\mathscr D)$ and the norm $\norm{\cdot}_{L_{p,p}(\mathscr D)}$ by $\norm{\cdot}_{L_p(\mathscr D)}$.
Throughout the rest of the paper, we shall adopt the usual summation convention over repeated indices.

We assume that the coefficients of $P$ are defined in $\mathscr D=(-\infty,\infty) \times \Omega$ and satisfy the following conditions which will be referred to collectively as (H).
\begin{enumerate}[({H}1)]
\item
There exists a constant $\nu \in (0,1)$ such that for all $\vec \xi=(\xi_1,\ldots, \xi_n) \in \bR^n$ and for all $(t,x) \in \mathscr D$, we have
\begin{equation*}				
\nu \abs{\vec \xi}^2\le \mathbf{A}\vec \xi\cdot \vec \xi= a^{ij}(t,x)\xi_i \xi_j\quad \text{and}\quad
\sum_{i,j=1}^{n}\, \abs{a^{ij}(t,x)}^2\le \nu^{-2}.
\end{equation*}
\item
$\vec b=(b^1,\ldots, b^n)$, $\vec c=(c^1,\ldots, c^n)$ are contained in some $L_{p,q}(\mathscr D)$  and $d$ is contained  in some $L_{p/2,q/2}(\mathscr D)$ where $p$ and $q$ are such that 
\begin{equation*}					
2\le p, q \le \infty\quad\text{and}\quad \frac{n}{p}+\dfrac{2}{q}=1.
\end{equation*}
There exists a constant $\Theta \ge 0$ such that
\begin{equation*}				
\norm{\vec b-\vec c}_{L_{p,q}(\mathscr D)} \le \Theta.
\end{equation*}
\item
The following inequalities hold in the sense of distributions:
\[
d-\Div\vec b\geq 0\quad\text{and}\quad \Div(\vec b-\vec c)\geq0.
\]
\end{enumerate}

It should be noted that we deal with the ``critical'' mixed norm spaces in (H2) and in the case when $p=n$ and $q=\infty$, it can be weakened to $(\vec b -\vec c) 1_{\mathscr D} \in L_\infty(\mathrm{BMO}^{-1})$;  see \eqref{eq_bmo_x}.
The condition (H3) allows us to obtain the ``global ''energy inequality and also ``scale invariant'' local boundedness estimate for weak solutions.

The goal of this paper is to show that if $P$ satisfies the condition (H), then there exists the Green's function $G(t,x,s,y)$ and it has the following Gaussian bound: there exist constants $C=C(n, \nu, p, \Theta)$ and $\kappa=\kappa(n, \nu, \Theta)>0$ such that for all $t, s$ satisfying $-\infty<s<t<+\infty$ and $x$, $y \in \Omega$, we have
\begin{equation}		\label{eq1013sun}
\abs{G(t,x,s,y)} \le \frac{C}{(t-s)^{\frac{n}{2}}} \exp \left\{- \frac{\kappa \abs{x-y}^2}{t-s} \right\}.
\end{equation}

We will  give some brief history regarding the Gaussian bounds for fundamental solutions of parabolic equations with measurable coefficients, starting with the case when there are no lower order terms present.
Since the groundbreaking work of Nash \cite{Nash}, where he established certain estimates of the fundamental solutions in proving H\"{o}lder continuity of weak solutions, there have been many important works in this field.
By employing the parabolic Harnack inequality of Moser \cite{Moser}, Aronson \cite{Aronson67} established two-sided Gaussian bounds for the fundamental solutions.
Fabes and Strook \cite{FS86} showed that the Nash's method could be used to prove Aronson's Gaussian bounds and as a consequence, they gave a new proof of Moser's parabolic Harnack inequality.

In the elliptic setting, Littman, Stampacchia, and Weinberger \cite{LSW} and Gr\"{u}ter and Widman \cite{GW82} studied Green's functions of elliptic equations in divergence form with measurable coefficients and showed that the Green's function $G(x,y)$ has a pointwise bound
\begin{equation}		\label{eq1014sun}
\abs{G(x,y)} \le C\abs{x-y}^{2-n}\quad (n\ge 3).
\end{equation}
Later, Hofmann and Kim \cite{HK07} gave an approach that also works for elliptic systems, where it is shown that the Green's function has pointwise bound \eqref{eq1014sun} if weak solutions of the elliptic system satisfy certain scale invariant H\"{o}lder continuity estimates.
Recently, Kim and Sakellaris \cite{KS19} studied Green's function of elliptic operators of the form
\begin{equation}			\label{eq1136sun}
Lu=-\Div (\mathbf{A}\nabla u+\vec{b} u)+\vec c\cdot\nabla u+du,
\end{equation}
where the principal coefficients $\mathbf{A}$ satisfy (H1), the lower order coefficients $\vec b$, $\vec c$, and $d$ satisfy (in the critical setting) the conditions that $\vec b$, $\vec c \in L_n$, $d \in L_{n/2}$, $d-\Div \vec b \ge 0$, and $d-\Div \vec c \ge 0$.
Assuming that $\Omega$ has a finite measure, they established pointwise bounds \eqref{eq1014sun} for the Green's function.

We would like to mention that our investigation is largely motivated by \cite{KS19}.
As it is well known, the Gaussian bound \eqref{eq1013sun} for the ``heat kernel'' of the elliptic operator $L$ yields the pointwise bound \eqref{eq1014sun} for the Green's function of $L$.
Therefore, in the elliptic context, our result says that if (H1) and (H3) hold and if $\vec b$, $\vec c \in L_n(\Omega)$ with $\norm{\vec b-\vec c}_{L_n} \le \Theta$ and $d \in L_{n/2}(\Omega)$, then the Green's function for the elliptic operator $L$ has the pointwise bound \eqref{eq1014sun} with constant $C=C(n, \nu, \Theta)$.
This gives a new proof for a result in \cite{KS19} dispensing with the assumption that $\abs{\Omega}<\infty$.

There are also many previous results in the literature regarding Gaussian bounds for fundamental solutions for parabolic equations with lower order terms.
To name a few, we mention \cite{Aronson68, QX1, QX2, Semenov, Zhang}.
However, there are very few in the literature dealing with global Gaussian bound \eqref{eq1013sun}.
For example, Aronson \cite{Aronson68} considered parabolic equations of the form \eqref{main para eq} with coefficients $\vec b$, $\vec c \in L_{p,q}$ and $d\in L_{p/2, q/2}$, where $2\le p, q \le \infty$ and $n/p+2/q<1$ (i.e. the subcritical case).
He obtained Gaussian bound for the fundamental solution without imposing the condition (H3) but the bound is not global in time.

For parabolic systems without lower order terms, Cho, Dong, and Kim \cite{CDK08} established the existence of Green's function in $(-\infty,\infty)\times \Omega$, under the assumption that weak solutions of the system satisfy certain scale invariant H\"{o}lder continuity estimate.
Recently, Dong and Kim \cite{DK18} extended the main result in \cite{CDK08} to parabolic systems with divergence free drift terms in the class of $L_\infty(\mathrm{BMO}^{-1})$.
In the scalar setting, the assumptions in \cite{DK18} read that $\vec b=0$, $\Div\vec c=0$, and $d\ge 0$, which obviously satisfy (H3).
As a matter of fact, in the proof of our main theorem, the only place where we strongly use the scalar property of weak solutions is in the the proof of Lemma~\ref{loc_bdd}, and as long as we have the local boundedness estimate in Lemma~\ref{loc_bdd}, essentially the same proof carries over to the systems.
Therefore, we recover the result in \cite{DK18} as a corollary. See Remark~\ref{rmk3.4}.

The organization of the paper is as follows.
In section~\ref{preliminary}, we introduce some notation and function spaces.
Then we prove the energy inequality which in turn implies the existence and uniqueness of weak solutions of Cauchy problems.
We present the main result in section~\ref{sec_main} and provide the proof in section~\ref{sec_proof}.
In section~\ref{sec7}, we prove the local boundedness estimate for weak solutions, which plays a key role in establishing the Gaussian bound.

\section{Preliminaries}\label{preliminary}
In this section, we first recall some frequently used notation and function spaces in \cite{LSU}. 
Then for reader's convenience, we give the definition of weak solutions to second order parabolic equations, and present some auxiliary estimates which will be used later.

\subsection{Notation and function spaces}
The adjoint operator $P^*$ of $P$ is defined by
\begin{equation}			\label{adj}
P^* u=-\partial_t u -\Div({\bf A}\tran\nabla u+\vec cu)+\vec b\cdot\nabla u+du,
\end{equation}
where ${\bf A}\tran$ is the transpose of $\bf A$.
Note that the coefficients ${\bf A}\tran=(a^{ji})$ satisfy the same ellipticity condition (H1).

We denote points in $\bR^{n+1}$ by $X=(t,x)$, $Y=(s,y)$, $X_0=(t_0,x_0)$, etc.
We define the ``parabolic distance'' between the points $X=(t,x)$ and $Y=(s,y)$ in $\bR^{n+1}$ as
\[
\abs{X-Y}=\max( \sqrt{\abs{t-s}},\abs{x-y}).
\]

For $U \subset \bR^{n+1}$, we write $U(t_0)$ for the set of all points $(t_0,x)$ in $U$ and $I(U)$ for the set of all $t$ such that $U(t)$ is nonempty.
We define
\begin{equation*}		
\oldnorm{u}_{U}:= \norm{Du}_{L_2(U)}+\esssup\limits_{t\in I(U)}\, \norm{u(t,\cdot)}_{L_2(U(t))}.
\end{equation*}
In the rest of this section, we restrict ourselves to the case when $U$ is a cylindrical domain $\mathscr D= (a,b) \times \Omega$ with $-\infty < a <b < +\infty$.
We define the lateral boundary $\partial_x \mathscr D$ and the parabolic boundary $\partial_p \mathscr D$ of $\mathscr D$ by
\[
\partial_x \mathscr D= [a,b]\times \partial\Omega\quad\text{and}\quad \partial_p \mathscr D= \partial_x \mathscr D \cup \set{t=a}\times \Omega,
\]
respectively.
We denote
\[
W_2^{1,0}(\mathscr D):=\{u: u, Du\in L_2(\mathscr D)\}, \quad 
W_2^{1,1}(\mathscr D):=\{u: u, \partial_t u, Du\in L_2(\mathscr D)\}.
\]

We define $V_2(\mathscr D)$ as the Banach space consisting of all elements of $W_2^{1,0}(\mathscr D)$ having a finite norm
\[
\norm{u}_{V_2(\mathscr D)} := \oldnorm{u}_{\mathscr D}=\norm{Du}_{L_2(\mathscr D)}+  \norm{u}_{L_{2,\infty}(\mathscr D)}.
\]

We define $V_2^{1,0}(\mathscr D)$ as the Banach space consisting of all elements of $V_2(\mathscr D)$ which are continuous in $t$ in the norm of $L_2(\Omega)$, with the norm
\[
\norm{u}_{V_2^{1,0}(\mathscr D)} := \oldnorm{u}_{\mathscr D}=\norm{Du}_{L_2(\mathscr D)}+ \max_{a\leq t\leq b}\, \norm{u(t,\cdot)}_{L_2(\Omega)}.
\]
The space $V_2^{1,0}(\mathscr D)$ is obtained by completing the set $W_2^{1,1}(\mathscr D)$ in the norm of $V_2(\mathscr D)$.

We say that a function $u$ in $W_2^{1,0}(\mathscr D)$ vanishes on $S \subset \partial_x \mathscr D$ if $u$ is the limit in $W_2^{1,0}(\mathscr D)$ of functions from $C^{1,1}_c(\overline{\mathscr D}\setminus S)$, the set of all continuously differentiable functions with compact supports in $\overline{\mathscr D}\setminus S$.
We denote by $\mathring{W}_2^{1,0}(\mathscr D)$ the set of functions in $W_2^{1,0}(\mathscr D)$ that vanish on the lateral boundary $\partial_x \mathscr D$.
We define
\[
\mathring{V}_2(\mathscr D):=V_2(\mathscr D)\cap \mathring{W}_2^{1,0}(\mathscr D)\quad\text{and}\quad \mathring{V}_2^{1,0}(\mathscr D):=V_2^{1,0}(\mathscr D)\cap \mathring{W}_2^{1,0}(\mathscr D).
\]

\subsection{Embedding inequalities}
Let the exponents $\tilde p$ and $\tilde q$ satisfy \[\frac{n}{\tilde p}+\frac{2}{\tilde q}=\frac{n}{2}\] with 
\begin{equation}\label{cond q r embed}
\begin{cases}
\tilde p\in [2,\frac{2n}{n-2}] ,\quad \tilde q\in [2,\infty]&\quad \text{if }\;n\ge 3,\\
\tilde p\in [2,\infty), \quad \;\;\tilde q\in (2,\infty]&\quad \text{if }\;n=2,\\
\tilde p\in[2,\infty],\quad \;\;\tilde q\in[4,\infty]&\quad\text{if }\; n=1.
\end{cases}
\end{equation}
By a well-known embedding theorem, there exists a constant $\beta=\beta(n, \tilde p)$ such that 
\begin{equation}\label{embd_ineq}
\norm{u}_{L_{\tilde p,\tilde q}(\mathscr D)} \leq \beta \oldnorm{u}_{\mathscr D},\quad\forall u\in\mathring{V}_2(\mathscr D).
\end{equation}
We emphasize that the constant $\beta$ in \eqref{embd_ineq} is independent of $\mathscr D$.
In particular, if we take $\tilde p=\tilde q=2(n+2)/n$ in \eqref{cond q r embed}, then $\beta=\beta(n)$.
See \cite[pp.~74-75]{LSU}.

\subsection{Weak solutions}
Let $\mathscr D=(a,b)\times \Omega$, where $-\infty < a <b < +\infty$.
Let  $f \in L_{\tilde p', \tilde q'}(\mathscr D)$, where $\tilde p'$ and $\tilde q'$ are H\"older conjugates of $\tilde p$ and $\tilde q$, respectively, and $\tilde p$ and $\tilde q$ satisfy $\frac{n}{\tilde p}+\frac{2}{\tilde q}=\frac{n}{2}$ with ranges specified in \eqref{cond q r embed}.
We say that $u \in V_2(\mathscr D)$ is a weak solution of $Pu=f$ in $\mathscr D$ if for almost all $t_{1}\in (a, b)$ the identity
\begin{multline*}
I(t_1; u, \phi):=\int_\Omega u(t_1, x)\phi(t_1,x)\,dx -\int_a^{t_1}\!\!\!\int_{\Omega}u\phi_t\,dxdt+\int_a^{t_1}\!\!\!\int_\Omega (a^{ij}D_j  u+b^i u)D_i \phi \,dxdt\\
 + \int_a^{t_1}\!\!\!\int_\Omega (c^i D_i u\phi+du\phi)\,dxdt -\int_a^{t_1}\!\!\!\int_{\Omega}f\phi\,dxdt=0
\end{multline*}
holds for all $\phi\in C^{1,1}_c(\overline{\mathscr D}\setminus \partial_p \mathscr D)$.

For a given function $\psi_0 \in L_2(\Omega)$, we say that $u\in\mathring{V}_2^{1,0}(\mathscr D)$ is a weak solution of the problem
\begin{equation}\label{weak sol}
Pu=f\;\text{ in }\;\mathscr D,\quad  u(a,\cdot)=\psi_0\;\text{ on }\;\Omega,
\end{equation}
if for all  $t_1\in [a,b]$ the identity
\begin{equation}			\label{eq1500f}
I(t_1; u, \phi)=\int_\Omega \psi_0(x)\phi(a,x)\,dx
\end{equation}
holds for all $\phi \in C^{1,1}_c(\overline{\mathscr D}\setminus \partial_x \mathscr D)$.

For the adjoint operator $P^*$ given by \eqref{adj}, we similarly define weak solutions of $P^*u=f$ in $\mathscr D$ and weak solutions of the corresponding backward problem
\begin{equation}		\label{backward_prob}
P^*u=f\;\text{ in }\;\mathscr D,\quad  u(b,\cdot)=\psi_0\;\text{ on }\;\Omega.
\end{equation}

\subsection{Energy inequality}
Under the condition (H), we can derive the following ``global'' energy inequality.
\begin{lemma}			\label{energy_ineq}
Suppose the coefficients of the operator $P$ satisfy the condition (H).
Let $\mathscr D=(a, b)\times\Omega$, where $-\infty<a<b <+\infty$.
Let $f \in L_{(2n+4)/(n+4)}(\mathscr D)$ and $\psi_0 \in L_2(\Omega)$ be given.
If $u\in\mathring{V}_2^{1,0}(\mathscr D)$ is a weak solution of the problem \eqref{weak sol},
then we have
\begin{equation}		\label{energy estimate}
\oldnorm{u}_{\mathscr D}\leq C\left(\norm{\psi_0}_{L_2(\Omega)}+\norm{f}_{L_{(2n+4)/(n+4)}(\mathscr D)}\right),
\end{equation}
where $C$ is a constant depending only on $n$ and $\nu$.
The same estimate is true if $u\in\mathring{V}_2^{1,0}(\mathscr D)$ is a weak solution of the the corresponding backward problem \eqref{backward_prob}.
\end{lemma}

\begin{proof}
By taking $\phi=u_h$ in \eqref{eq1500f}, where $u_h$ is the Steklov average of $u$ (see \cite[\S III.2]{LSU}), integrating by parts, and taking $h\to 0$, we have for all  $t_1 \in [a,b]$ that
\begin{multline}				\label{eq1529w}
\frac12 \int_\Omega u^2(t_1,x)\,dx+\int_a^{t_1}\!\!\!\int_\Omega (a^{ij}D_ju+b^i u)D_i u \,dxdt + \int_a^{t_1}\!\!\!\int_\Omega (c^iD_i u u+du^2) dx dt \\
=\frac12 \int_\Omega \psi_0(x) u(a,x)\,dx+ \int_a^{t_1}\!\!\!\int_\Omega f u\,dxdt.
\end{multline}
Since the condition (H3) implies that
\[
\int_a^{t_1}\!\!\!\int_\Omega du^2+2 b^i u D_i u \ge 0\quad \text{and}
\quad
-\int_a^{t_1}\!\!\!\int_\Omega (b^i-c^i) uD_iu \ge 0,
\]
it follows from \eqref{eq1529w} and the condition (H1) that
\begin{equation}			\label{ineq energy1}
\frac12 \int_\Omega u^2(t_1,x)\,dx+ \nu \int_a^{t_1}\!\!\!\int_\Omega \abs{Du}^2\,dx dt \le \int_a^{t_1}\!\!\!\int_\Omega fu\,dxdt+ \frac12 \int_\Omega \psi_0(x) u(a,x)\,dx.
\end{equation}
By H\"older's inequality and the embedding \eqref{embd_ineq}, we have
\begin{equation}			\label{eq1432f}
\int_a^{t_1}\!\!\! \int_\Omega f u\,dxdt \le \norm{f}_{L_{(2n+4)/(n+4)}(\mathscr D)}\norm{u}_{L_{(2n+4)/n}(\mathscr D)} \le \beta \norm{f}_{L_{(2n+4)/(n+4)}(\mathscr D)}\oldnorm{u}_{\mathscr D},
\end{equation}
where $\beta=\beta(n)$. 
Now, the estimate \eqref{energy estimate} follows from the standard argument involving Young's inequality.

For the corresponding backward problem \eqref{backward_prob}, similar to \eqref{ineq energy1}, we have 
\[
\frac12 \int_\Omega u^2(t_1,x)\,dx+ \nu \int_{t_1}^b\!\!\!\int_\Omega \abs{Du}^2\,dx dt \le \int_{t_1}^b\!\!\!\int_\Omega fu\,dxdt+ \frac12 \int_\Omega \psi_0(x) u(b,x)\,dx
\]
for all $t_1 \in [a,b]$.
Hence, the energy inequality \eqref{energy estimate} is also valid for the backward problem \eqref{backward_prob}.
\end{proof}

\subsection{Existence and uniqueness of weak solutions}
With the energy inequality \eqref{energy estimate} available, we can construct a weak solution of the problem \eqref{weak sol} by Galerkin's method.
Uniqueness is also a consequence of the energy inequality.
See \cite[\S III.4]{LSU} for the details.
We state these observations in the following lemma for the reference.

\begin{lemma}\label{lem_uniq}
Suppose the coefficients of the operator $P$ satisfy the condition (H).
Let $\mathscr D=(a, b)\times\Omega$, where $-\infty<a<b <+\infty$.
Let $f \in L_{(2n+4)/(n+4)}(\mathscr D)$ and $\psi_0 \in L_2(\Omega)$ be given.
Then there exists a unique weak solution $u\in\mathring{V}_2^{1,0}(\mathscr D)$ of the problem \eqref{weak sol}.
The same is true for the backward problem \eqref{backward_prob}.
\end{lemma}

We note that the condition that $f \in L_{(2n+4)/(n+4)}(\mathscr D)$ in Lemmas~\ref{energy_ineq} and \ref{lem_uniq} can be replaced by $f \in L_{\tilde p', \tilde q'}(\mathscr D)$, where $\tilde p'$ and $\tilde q'$ are H\"older conjugates of $\tilde p$ and $\tilde q$, respectively, and $\tilde p$ and $\tilde q$ satisfy $\frac{n}{\tilde p}+\frac{2}{\tilde q}=\frac{n}{2}$ with ranges specified in \eqref{cond q r embed}.
This is because the inequality \eqref{eq1432f} remains valid with $L_{\tilde p',\tilde q'}$ norm of $f$.
However, we do not use this fact in the paper.

If $\mathscr D=(a,\infty)\times \Omega$ and $f \in L_{(2n+4)/(n+4)}(\mathscr D)$, then by letting $b\to \infty$ in Lemmas~\ref{energy_ineq} and \ref{lem_uniq}, we can say that $u$ is the weak solutions in $\mathring{V}^{1,0}_2(\mathscr D)$ of the problem \eqref{weak sol}.

\subsection{local boundedness property}
The following lemma says that we have ``scale-invariant'' local boundedness property for weak solutions of $Pu=f$ in $Q^{-}_r(X_0)$ that vanish on $S_r^{-}(X_0)$, where
\begin{equation}			\label{eq1717f}
\begin{aligned}
Q_r^{-}(X_0)&=(t_0-r^2,t_0)\times (B_r(x_0) \cap \Omega),\\
S_r^{-}(X_0)&=(t_0-r^2,t_0)\times (B_r(x_0) \cap \partial\Omega),
\end{aligned}
\end{equation}
and $-\infty <t_0<+\infty$ and $x_0 \in \Omega$.
When dealing with the adjoint operator $P^{*}$, we replace $Q^{-}_r(X_0)$ and $S^{-}_r(X_0)$ with $Q^{+}_r(X_0)$ and $S^{+}_r(X_0)$, where
\begin{equation}			\label{eq1718f}
\begin{aligned}
Q_r^{+}(X_0)&=(t_0, t_0+r^2)\times (B_r(x_0) \cap \Omega),\\
S_r^{+}(X_0)&=(t_0, t_0+r^2)\times (B_r(x_0) \cap \partial\Omega).
\end{aligned}
\end{equation}

\begin{lemma}\label{loc_bdd}
Suppose the coefficients of the operator $P$ satisfy the condition (H).
Let $Q_r^{-}=Q_r^{-}(X_0)$ and $S_r^{-}=S_r^{-}(X_0)$.
If $u\in V_2(Q_r^{-})$ is a weak solution of $Pu=f$ in $Q_r^{-}$ vanishing on $S_r^{-}$, where $f\in L_{\infty}(Q_r^{-})$, then we have
\begin{equation}			\label{loc_bdd_est}
\norm{u}_{L_{\infty}(Q^{-}_{r/2})}\leq N_0\left( r^{-\frac{n+2}{2}} \norm{u}_{L_2(Q^{-}_r)}+r^2\norm{f}_{L_{\infty}(Q^{-}_r)}\right),
\end{equation}
where $N_0$ is a constant that depends only on $n$, $\nu$, $p$, and $\Theta$.
The corresponding statement is valid for the weak solution of $P^{*}u=f$ in $Q^{+}_r(X_0)$ vanishing on $S^{+}_r(X_0)$.
\end{lemma}
We emphasize that the constant $N_0$ in the lemma is independent of $r$.
The proof will be given in section~\ref{sec7}.

\section{Main results}			\label{sec_main}
\begin{theorem}			\label{thm_main}
Suppose the the coefficients of operator $P$ satisfy the condition (H) and let $\mathscr D=(-\infty,\infty) \times\Omega$.
Then, there exists a unique Green's function $G(X,Y)=G(t,x,s,y)$ on $\mathscr D \times \mathscr D$ which satisfies $G(t,x,s,y)\equiv0$ for $t<s$, and has the following property:
For any $\psi_0 \in L_2(\Omega)$, the function $u$ given by
\begin{equation}\label{idetity 00}
u(t,x):=\int_{\Omega}G(t,x,s,y)\psi_0(y)\ dy \quad (t>s,\;\; x\in\Omega)
\end{equation}
is the unique weak solution in $\mathring{V}_2^{1,0}((s,\infty)\times\Omega)$ of the problem
\[
Pu=0\;\text{ in }\;(s,\infty)\times \Omega,\quad  u(s,\cdot)=\psi_0\;\text{ on }\;\Omega.
\]
Moreover, the Green's function satisfies the following Gaussian bound:
For all $t>s$ and $x$, $y\in \Omega$, we have
\begin{equation}\label{Gaussian bound}
\abs{G(t,x,s,y)}\leq\frac{C}{(t-s)^{\frac{n}{2}}}\exp\left\{-\frac{\kappa\abs{x-y}^2}{t-s}\right\},
\end{equation}
where $C=C(n, \nu, p, \Theta)$ and $\kappa=\kappa(n,\nu,\Theta)$ are positive constants.
\end{theorem}

\begin{corollary}
Let $\Omega$ be an open connected set in $\bR^n$ with $n\ge 3$.
Suppose the coefficients of elliptic operator $L$ in \eqref{eq1136sun} satisfy the condition (H1) and (H3). 
In place of (H2), assume that $\vec b$, $\vec c \in L_n(\Omega)$, $d \in L_{n/2}(\Omega)$, and that $(\vec b-\vec c)1_\Omega \in \mathrm{BMO}^{-1}$, that is, there are functions $\Phi^{ij}$ in $\bR^n$ and a positive constant $\Theta$ such that
\[
(b^i-c^i)1_\Omega= D_j \Phi^{ij},\quad \sum_{i,j=1}^n\,\norm{\Phi_{ij}}_{\mathrm{BMO}(\bR^n)}^2 \le \Theta^2.
\]
Then there exists the Green's function $G(x,y)$ on $\Omega \times \Omega$ and it has the bound 
\[
\abs{G(x,y)} \le C\abs{x-y}^{2-n},
\]
where $C=C(n, \nu, \Theta)$.
\end{corollary}
\begin{proof}
Let $K_t(x,y)=\tilde G(t,x,0,y)$, where $\tilde G(t,x,s,y)$ is the Green's function for the operator $P$ with time independent coefficients; as mentioned in the introduction, when $p=n$, the condition (H2) can be relaxed to the weaker condition \eqref{eq_bmo_x}.
Let
\[
G(x,y)=\int_0^\infty K_t(x,y)\,dt.
\]
Then, it is known that $G(x,y)$ becomes the Green's function for the operator $L$; see, e.g., \cite{DK09}.
From the Gaussian bound \eqref{Gaussian bound}, it follows
\[
\abs{G(x,y)} \le \int_0^\infty \abs{K_t(x,y)}\,dt \le C \int_0^\infty t^{-n/2} e^{-\kappa \abs{x-y}^2/t}\,dt \le C \abs{x-y}^{2-n}.\qedhere
\]
\end{proof}

\begin{remark}			\label{rmk3.4}
Consider the second-order parabolic systems of divergence form
\[
P_{ij}u^j=\partial_t u^i - D_{\alpha}(a_{ij}^{\alpha\beta}(t,x) D_{\beta}u^j+b_{ij}^\alpha(t,x)u^j)
+ c_{ij}^{\alpha}(t,x) D_\alpha u^j + d_{ij}(t,x)u^{j},\quad i=1,\ldots, m,
\]
where the coefficients satisfy the following conditions analogous to (H).
\begin{enumerate}[({H}1')]
\item
There exists a constant $\nu \in (0,1)$ such that for all $(t,x) \in \mathscr D$, we have
\[
\nu \sum_{i=1}^m \sum_{\alpha=1}^n \, \abs{\xi^i_\alpha}^2 \le a^{\alpha\beta}_{ij}(t,x) \xi^i_\alpha \xi^j_\beta\quad\text{and}\quad
\sum_{i,j=1}^m \sum_{\alpha, \beta=1}^n \,\abs{a^{\alpha\beta}_{ij}(t,x)}^2 \le \nu^{-2}.
\]
\item
$\mathbf{b^\alpha}=(b^\alpha_{ij})$, $\,\mathbf{c^\alpha}=(c^\alpha_{ij})$ are symmetric and belong to some $L_{p,q}(\mathscr D)$, and $\mathbf{d}=(d_{ij})$ is contained in some $L_{p/2,q/2}(\mathscr D)$, where $p$ and $q$ are such that 
\begin{equation*}					
2\le p, q \le \infty\quad\text{and}\quad \frac{n}{p}+\dfrac{2}{q}=1.
\end{equation*}
There exists a constant $\Theta > 0$ such that
\begin{equation*}				
\sum_{\alpha=1}^n \,\norm{\mathbf{b^\alpha}-\mathbf{c^\alpha}}_{L_{p,q}(\mathscr D)}^2 \le \Theta^2.
\end{equation*}
\item
The following inequalities hold in the sense of distributions:
\[
\mathbf{d}-D_\alpha \mathbf{b}^\alpha \geq 0\quad\text{and}\quad
D_\alpha(\mathbf{b}^\alpha-\mathbf{c}^\alpha)\geq0,
\]
Here, $\mathbf{M}\ge 0$ for a matrix $\mathbf{M}$ means that $\mathbf{M}\vec \xi \cdot \vec \xi \ge 0$ for all $\vec \xi \in \bR^n$.
\end{enumerate}
Also, we assume that local boundedness property holds for the operator $P_{ij}$ and its adjoint operator.
Then the conclusion of Theorem~\ref{thm_main} is true. See \cite{DK18}.
\end{remark}

\section{Proof of Theorem \ref{thm_main}}			\label{sec_proof}
\subsection{Construction of the Green's function}
The proof for construction of Green's function is a modification of that given in \cite{DK18, CDK08}.
For reader's convenience we present main steps here.
Let $Y=(s,y)\in \mathscr D$.
For $\varepsilon>0$, fix $a \in(-\infty, s-\varepsilon^2)$ and $b \in (s, \infty)$.
We consider the problem
\begin{equation}		\label{Diri para prb}
Pv=\frac{1}{\abs{Q^{-}_\varepsilon(Y)}}1_{Q^{-}_\varepsilon(Y)}\;\text{ in }\;(a,b)\times \Omega,\qquad
v(a,\cdot)=0 \;\text{ on }\;\Omega,
\end{equation}
where $1_{Q^{-}_\varepsilon(Y)}$ is a characteristic function and 
$Q_\varepsilon^{-}(Y)$ is as defined in \eqref{eq1717f}.

By Lemma~\ref{lem_uniq}, there exists a unique weak solution $v_\varepsilon=v_{\varepsilon;Y}\in \mathring{V}_2^{1,0}((a,b)\times \Omega)$ of the problem \eqref{Diri para prb}.
Furthermore, by using the uniqueness, we find that the solution $v_\varepsilon$ does not depend on  $a$ or $b$, and we may extend $v_\varepsilon$ to entire $\mathscr D$ by setting $v_\varepsilon\equiv 0$ in $(-\infty,a)\times \Omega$ and letting $b \to \infty$.
Then by the energy estimate \eqref{energy estimate}, we have
\begin{equation}			\label{eq1644m}
\oldnorm{v_\varepsilon}_{\mathscr D}\le \abs{Q^{-}_\varepsilon(Y)}^{-\frac{n}{2(n+2)}}.
\end{equation}
We define the ``approximate'' Green's function $G^{\varepsilon}(\cdot,Y)$ for $P$ in $\mathscr D$ by
\[
G^{\varepsilon}(\cdot,Y)=v_\varepsilon.
\]

Next, for  $f\in C^\infty_c(\mathscr D)$, choose a number $b$ such that $f\equiv 0$ in $[b,\infty)\times \Omega$.
For any $a<b$, consider the backward problem
\begin{equation}		\label{Diri backpara prb}
P^{*}u=f \;\text{ in }\;(a,b)\times \Omega,\qquad
u(b,\cdot)=0 \;\text{ on }\; \Omega.
\end{equation}
By Lemma~\ref{lem_uniq} again, we obtain a unique weak solution $u\in \mathring{V}_2^{1,0}((a,b )\times\Omega)$ of the problem \eqref{Diri backpara prb}.
Again, we may extend $u$ to entire $\mathscr D$ by setting $u \equiv 0$ in $(b,\infty)\times \Omega$ and letting $a \to -\infty$.
The energy inequality \eqref{energy estimate} then tells us that
\begin{equation}			\label{eq2.00}
\oldnorm{u}_{\mathscr D} \le C\norm{f}_{L_{(2n+4)/(n+4)}(\mathscr D)}.
\end{equation}

Notice from \eqref{Diri para prb} and \eqref{Diri backpara prb} that we have
\begin{equation}			\label{eq2.17}
\int_{\mathscr D} G^\varepsilon(\cdot, Y) f=\fint_{Q^{-}_\varepsilon(Y)}u.
\end{equation}

Now, we assume that $f$ is supported in $Q^{+}_R(X_0)$, where it is defined in \eqref{eq1718f}.
By Lemma~\ref{loc_bdd} combined with \eqref{eq2.00} and \eqref{embd_ineq}, we have
\begin{equation} \label{eq2.19}
\norm{u}_{L_\infty(Q_{R/2}^+(X_0))} \le C R^{2} \norm{f}_{L_\infty(Q_R^{+}(X_0))}.
\end{equation}
If $Q^{-}_\varepsilon(Y)\subset Q^+_{R/2}(X_0)$, then \eqref{eq2.17} together with
\eqref{eq2.19} yields
\[
\Abs{\int_{Q^{+}_R(X_0)}G^\varepsilon(\cdot, Y) f \,} \le \fint_{Q^{-}_\varepsilon(Y)}\abs{u}\le CR^{2} \norm{f}_{L_{\infty}(Q^{+}_R(X_0))}.
\]
By duality, it follows that if $Q^{-}_\varepsilon(Y)\subset Q^{+}_{R/2}(X_0)$, then
\[
\norm{G^\varepsilon(\cdot, Y)}_{L_1(Q^{+}_R(X_0))}\le CR^{2}.
\]

Therefore, the same proof of \cite[Lemma~3.6]{DK18} yields the following lemma.

\begin{lemma}\label{lem G vare}
Let $X=(t,x)$, $Y=(s,y) \in \mathscr D$ with $X \neq Y$.
Then we have
\[
\abs{G^{\varepsilon}(X,Y)} \leq C\abs{X-Y}^{-n}, \quad \forall \varepsilon\leq\tfrac{1}{3}\abs{X-Y},
\]
where $C=C(n, \nu, p, \Theta)$.
\end{lemma}

For $\rho$ and $R$ satisfying  $\varepsilon<\rho<R$, let $\zeta: \bR^{n+1} \to [0,1]$ be a smooth function satisfying $\zeta(X)= 0$ for $\abs{X-Y}<\rho$,  $\zeta(X)=1$ for $\abs{X-Y}\ge R$, and
\begin{equation*}		
\max\left(\norm{D\zeta}_{L_\infty}^2,\; \norm{D^2\zeta}_{L_\infty}, \;\norm{\partial_t \zeta}_{L_\infty} \right) \le \frac{4}{(R-\rho)^2}.
\end{equation*}
Recall that $v_\varepsilon \in \mathring{V}^{1,0}_2(\mathscr D)$ and it satisfies \eqref{Diri para prb}.
Testing the equation with $\zeta^2 v_\varepsilon$, letting $a\to -\infty$, $b\to \infty$, and using that $d-\Div \vec b \ge 0$ in (H3), we obtain
\begin{multline*}
\frac12 \int_{\Omega} \zeta^2 v_\varepsilon^2(t,x)\,dx+\int_{-\infty}^t \!\int_\Omega \zeta^2 a^{ij}D_j v_\varepsilon D_i v_\varepsilon \,dxdt +\int_{-\infty}^t \!\int_\Omega 2 a^{ij} \zeta D_j v_\varepsilon D_i \zeta v_\varepsilon \,dxdt\\
\le \int_{-\infty}^t \!\int_\Omega \zeta\partial_t \zeta v_\varepsilon^2\,dxdt+\int_{-\infty}^t \!\int_{\Omega} \zeta^2 v_\varepsilon (b^i-c^i) D_iv_\varepsilon\,dxdt
\end{multline*}
for all $t$.
By the assumption that $\Div (\vec b-\vec c) \ge 0$ in (H3), we have
\begin{equation*}			
\int_{-\infty}^t \!\int_\Omega \zeta^2 v_\varepsilon (b^i-c^i)D_i v_\varepsilon\,dxdt \le -\int_{-\infty}^t \!\int_\Omega \zeta (b^i-c^i) v_\varepsilon^2 D_i\zeta\,dxdt.
\end{equation*}
Then by the condition (H1) and Young's inequality, we obtain
\begin{multline}			\label{eq1455sun}
\frac12 \norm{\zeta v_\varepsilon}_{L_{2,\infty}(\mathscr D)}^2+ \frac{3\nu}{4} \int_{\mathscr D} \zeta^2 \abs{Dv_\varepsilon}^2\,dxdt \\
\le 2 \int_{\mathscr D} \abs{\partial_t \zeta} v_\varepsilon^2\,dxdt+ \frac{8}{\nu^3}\int_{\mathscr D} \abs{D\zeta}^2 v_\varepsilon^2\,dxdt - 2\int_{\mathscr D}\zeta (b^i-c^i) v_\varepsilon^2 D_i \zeta\,dxdt.
\end{multline}

In the case when $p>n$ so that $q<\infty$, we use H\"{o}lder's inequality, the embedding \eqref{embd_ineq}, and Young's inequality, to find that
\begin{align}		\label{eq1801sun}
-2\int_{\mathscr D} \zeta (b^i-c^i)v_\varepsilon^2 D_i \zeta\,dxdt &\le 2\norm{\vec b-\vec c}_{L_{p,q}(\mathscr D)}\,\norm{\zeta v_\varepsilon}_{L_{2p/(p-2),2q/(q-2)}(\mathscr D)}\,\norm{D\zeta v_\varepsilon}_{L_2(\mathscr D)} \nonumber\\
&\le 2\beta \Theta \oldnorm{\zeta v_\varepsilon}_{\mathscr D}\,\norm{D\zeta v_\varepsilon}_{L_2(\mathscr D)} \nonumber\\
&\le \frac{\nu}{8} \oldnorm{\zeta v_\varepsilon}_{\mathscr D}^2+\frac{8\beta^2 \Theta^2}{\nu}\, \norm{D\zeta v_\varepsilon}_{L_2(\mathscr D)}^2.
\end{align}
Note that
\begin{align}		\label{eq1741f}
\oldnorm{\zeta v_\varepsilon}_{\mathscr D}^2 &\le 2 \norm{\zeta v_\varepsilon}_{L_{2,\infty}(\mathscr D)}^2+ 2\norm{D(\zeta v_\varepsilon)}_{L_2(\mathscr D)}^2	\nonumber\\
&\le 2 \norm{\zeta v_\varepsilon}_{L_{2,\infty}(\mathscr D)}^2+ 4\norm{D\zeta v_\varepsilon}_{L_2(\mathscr D)}^2+ 4\norm{\zeta Dv_\varepsilon}_{L_2(\mathscr D)}^2.
\end{align}
Since $0<\nu<1$, we get from \eqref{eq1455sun}, \eqref{eq1801sun}, and \eqref{eq1741f} that
\[
\norm{\zeta v_\varepsilon}_{L_{2,\infty}(\mathscr D)}^2+ \int_{\mathscr D} \zeta^2 \abs{Dv_\varepsilon}^2\,dxdt \le C \int_{\mathscr D} \left(\abs{\partial_t \zeta}+\abs{D\zeta}^2\right) v_\varepsilon^2\,dxdt.
\]
Then by using \eqref{eq1741f} again, taking $\rho=\frac12 r$ and $R=r$ for $\zeta$, and using Lemma~\ref{lem G vare}, we get
\begin{align*}	
\oldnorm{v_\varepsilon}_{\mathscr D\cap \set{X: \abs{X-Y}\ge r}}^2\le \oldnorm{\zeta v_\varepsilon}_{\mathscr D}^2 & \le \frac{C}{r^2} \int_{\mathscr D \cap \set{X: \frac12 r\le \abs{X-Y}\le r}} v_\varepsilon^2\,dxdt\\
 &\le  \frac{C}{r^2} \int_{\set{X: \frac12 r\le \abs{X-Y}\le r}} \abs{X-Y}^{-2n}\,dX \le \frac{C}{r^n}\le C \abs{Q_r^{-}(Y)}^{-\frac{n}{n+2}}
\end{align*}
provided $6 \varepsilon \le r$.
If $r>6\varepsilon$, then thanks to \eqref{eq1644m}, the same inequality is obviously true.
Therefore, we have
\begin{equation}			\label{eq1455m}
\oldnorm{v_\varepsilon}_{\mathscr D\cap \set{X: \abs{X-Y}\ge r}}^2 \le C \abs{Q_r^{-}(Y)}^{-\frac{n}{n+2}},\quad \forall r>0,
\end{equation}
which corresponds to \cite[(3.20)]{CDK08}.
With the uniform estimate \eqref{eq1455m} at hand, we may invoke the same compactness argument as presented in \cite[Section 3.3]{CDK08} and obtain a Green's function $G(\cdot,Y)$ from the family $\set{v_\varepsilon}=\set{G^\varepsilon(\cdot ,Y)}$ in the case when $p>n$.

In the case when $p=n$ and $q=\infty$, we use the following facts.
\begin{enumerate}[1.]
\item
There exist functions $\Phi^{ij}$ on $\bR^{n+1}$ satisfying
\begin{equation}	\label{eq_bmo_x}
(b^i -c^i)1_{\mathscr D} =D_j  \Phi^{ij}\,\quad \sup_{t\in\bR}\,\sum_{i,j=1}^{n} \,\norm{\Phi^{ij}(t,\cdot)}_\mathrm{BMO(\bR^n)}^2\leq C_n \norm{\vec b-\vec c}_{L_{n,\infty}}^2 \le C_n \Theta^2,
\end{equation}
where $C_n$ is a constant which depends only on $n$.
\item
For $f$, $g \in W^1_2(\bR^n)$ and $j=1,\ldots, n$, we have
\begin{equation}	\label{ccpi}
\norm{D_j(fg)}_{\mathscr{H}^1(\bR^n)}\leq C_n \left( \norm{Df}_{L_2(\bR^n)}\norm{g}_{L_2(\bR^n)}+\norm{f}_{L_2(\bR^n)}\norm{Dg}_{L_2(\bR^n)}\right),
\end{equation}
where $\norm{\cdot}_{\mathscr{H}^1(\bR^n)}$ denotes the Hardy norm.
\end{enumerate}
We note that \eqref{eq_bmo_x} is a consequence of the embedding $L_n(\bR^n) \hookrightarrow  \mathrm{BMO}^{-1}(\bR^n)$. See, e.g., \cite{KT01}.
Estimates of type \eqref{ccpi} are originally due to Coifman et al. \cite{CLMS} and usually referred to as ``compensated compactness''. 
See \cite[Proposition~3.2]{QX1} for the proof of \eqref{ccpi}.

By setting $v_\varepsilon(t,\cdot)=0$ outside $\Omega$ and applying \eqref{eq_bmo_x}, we  get
\begin{multline}				\label{eq1903sun}
-2\int_\Omega \zeta (b^i-c^i) v_\varepsilon^2 D_i \zeta
=2\int_\Omega \Phi^{ij}D_j  (\zeta v_\varepsilon^2 D_i \zeta)
\le 2 \norm{\Phi^{ij}}_{\mathrm{BMO}(\bR^n)}\,\norm{D_j (\zeta v_\varepsilon^2 D_i \zeta)}_{\mathscr{H}^{1}(\bR^n)} \\
\le C_n \Theta \left(\sum_{i,j=1}^n \norm{D_j(\zeta v_\varepsilon^2 D_i \zeta)}_{\mathscr{H}^1(\bR^n)}^2 \right)^{\frac12}=:\mathrm{RHS}.
\end{multline}
Fix a smooth function $\tilde \zeta:\bR^{n+1} \to [0,1]$ such that $\tilde \zeta(X)=1$ for $\abs{X-Y}<R$,  $\tilde \zeta(X)=0$ for $\abs{X-Y} \ge 2R$, and $\norm{D\tilde \zeta}_{L_\infty}\le 2/R$.
Then, since
\[
\zeta v_\varepsilon^2 D_i \zeta= \tilde \zeta \zeta v_\varepsilon^2 D_i \zeta,
\]
by taking $f=\tilde \zeta \zeta v_\varepsilon$ and $g=v_\varepsilon D_i \zeta$ in \eqref{ccpi}, and using Young's inequality, the right hand side of \eqref{eq1903sun}  is bounded by
\begin{align*}
\mathrm{RHS}&\le C \left\{ \left(\norm{D(\tilde \zeta \zeta)v_\varepsilon}_{L_2}+\norm{\tilde\zeta \zeta Dv_\varepsilon}_{L_2}\right) \norm{v_\varepsilon D\zeta}_{L_2} +\norm{\tilde \zeta \zeta v_\varepsilon}_{L_2} \left( \norm{D\zeta\tran Dv_\varepsilon}_{L_2}+\norm{v_\varepsilon D^2\zeta}_{L_2}\right)\right\}\\
& \le C \norm{D(\tilde \zeta \zeta) v_\varepsilon}_{L_2}^2 + C \norm{D \zeta v_\varepsilon}_{L_2}^2 + \frac{\nu}{4}\norm{\tilde \zeta \zeta Dv_\varepsilon}_{L_2}^2\\
&\qquad +\frac{C}{(R-\rho)^2}\norm{\tilde \zeta \zeta v_\varepsilon}_{L_2}^2 +\frac{\nu(R-\rho)^2}{16} \norm{D\zeta\tran D v_\varepsilon}_{L_2}^2 + (R-\rho)^2 \norm{D^2 \zeta v_\varepsilon}_{L_2}^2,
\end{align*}
where $C=C(n, \nu, \Theta)$.
Then, by integrating \eqref{eq1903sun} with respect to $t$ over $(-\infty, \infty)$, and using the properties of $\zeta$ and $\tilde \zeta$, we have
\begin{multline}			\label{eq1737sun}
-2\int_{\mathscr D} \zeta (b^i-c^i) v_\varepsilon^2 D_i \zeta \le \frac{\nu}{4} \int_{\mathscr D} \zeta^2 \abs{Dv_\varepsilon}^2\\
+ \frac{C}{(R-\rho)^2} \int_{\set{X: \rho \le \abs{X-Y}\le 2R}\cap \mathscr D} v_\varepsilon^2
+\frac{\nu}{4} \int_{\set{X: \rho \le \abs{X-Y}\le R}\cap \mathscr D} \abs{D v_\varepsilon}^2.
\end{multline}
Putting \eqref{eq1737sun} back to \eqref{eq1455sun} and using the properties of $\zeta$, we obtain
\begin{multline}			\label{eq2139m}
\frac12 \norm{\zeta v_\varepsilon}_{L_{2,\infty}(\mathscr D)}^2+ \frac{\nu}{2} \int_{\mathscr D} \zeta^2 \abs{Dv_\varepsilon}^2 \\
\le \frac{C}{(R-\rho)^2} \int_{\set{X: \rho \le \abs{X-Y}\le 2R}\cap \mathscr D} v_\varepsilon^2
+\frac{\nu}{4} \int_{\set{X: \rho \le \abs{X-Y}\le R}\cap \mathscr D} \abs{D v_\varepsilon}^2.
\end{multline}
In particular, \eqref{eq2139m} implies that
\[
\int_{\set{X:  \abs{X-Y} \ge R}\cap \mathscr D} \abs{Dv_\varepsilon}^2 \\
\le  \frac{C}{(R-\rho)^2} \int_{\set{X: \rho \le \abs{X-Y}\le 2R}\cap \mathscr D} v_\varepsilon^2 +\frac12 \int_{\set{X: \abs{X-Y}\ge \rho}\cap \mathscr D} \abs{D v_\varepsilon}^2.
\]
Since the above inequality is true for all $\rho$ and $R$ satisfying $\varepsilon<\rho<R$, a well-known iteration argument yields (see \cite[Lemma~5.1]{Giaquinta93}) that for any $r$ satisfying $\varepsilon < r< \infty$, we have 
\[
\int_{\set{X: \abs{X-Y}\ge 2r}\cap \mathscr D} \abs{Dv_\varepsilon}^2 \\
\le \frac{C}{r^2} \int_{\set{X: r \le \abs{X-Y}\le 4r}\cap \mathscr D} v_\varepsilon^2.
\]
Then, by taking $\rho=2r$ and $R=4r$, we get from \eqref{eq2139m} and Lemma~\ref{lem G vare} that
\begin{align*}
\oldnorm{v_\varepsilon}_{\mathscr D\cap \set{X: \abs{X-Y}\ge 4r}}^2 \le \oldnorm{\zeta v_\varepsilon}_{\mathscr D}^2 &\le \frac{C}{r^2} \int_{\set{X: 2r \le \abs{X-Y}\le 8r}\cap \mathscr D} v_\varepsilon^2 + \int_{\set{X: \abs{X-Y} \ge 2r}\cap \mathscr D} \abs{D v_\varepsilon}^2\\
&\le \frac{C}{r^2} \int_{\set{X: r \le \abs{X-Y}\le 8r}\cap \mathscr D} v_\varepsilon^2\\
& \le \frac{C}{r^2} \int_{\set{X: r \le \abs{X-Y}\le 8r}} \abs{X-Y}^{-2n}\,dX  \le \frac{C}{r^n} \le C \abs{Q_r(Y)}^{-\frac{n}{n+2}}
\end{align*}
provided that $3 \varepsilon \le r$. 
Again, thanks to \eqref{eq1644m}, we get the uniform estimate \eqref{eq1455m}, which allows us to construct a Green's function $G(\cdot,Y)$ out of the family  $\set{G^\varepsilon(\cdot ,Y)}$ in the case when $p=n$.

Also, by parallel reasonings, we can construct a Green's function $G^{*}(X,Y)$ for the adjoint operator $P^{*}$.
We refer to \cite[Section 3.5]{CDK08} for the proof of the representation formula  \eqref{idetity 00}, which also shows the uniqueness of Green's function.
This completes the proof of the first part of Theorem \ref{thm_main}.

\subsection{Gaussian estimates}
We now prove the Gaussian estimate \eqref{Gaussian bound}.
\subsubsection{Case when $p=n$}
In the case when $p=n$ and $q=\infty$, we follow the argument in \cite{DK18}, which is an adaptation of the techniques in \cite{Davies, CDK08, HK04}, to obtain Gaussian bound \eqref{Gaussian bound}.
Here, we shall make strong use of \eqref{eq_bmo_x} and \eqref{ccpi}.

Now, let $\psi:\bR^n \rightarrow\bR$ be a bounded $C^2$ function satisfying
\begin{equation*}		
\abs{D\psi} \leq\gamma_1,\quad \abs{D^2\psi} \leq\gamma_2,
\end{equation*}
for some positive constants $\gamma_1$ and $\gamma_2$ to be fixed later. For $t>s$, we define an operator $P_{s\rightarrow t}^{\psi}$ on $L_2(\Omega)$ as follows.
For a given $f\in L_2(\Omega)$, let $u\in\mathring{V}_2^{1,0}((s,\infty)\times\Omega )$ be the weak solution of the problem
\begin{equation}\label{prob Diri}
\begin{cases}
Pu=0,\\
u(s,\cdot)=e^{-\psi}f.
\end{cases}
\end{equation}
Then we define $P_{s \to t}^\psi f(x):=e^{\psi(x)}u(t,x)$.
It follows from \eqref{idetity 00} that
\begin{equation}\label{formula Pf}
P_{s\to t}^\psi f(x)=e^{\psi(x)}\int_{\Omega}G(t,x,s,y)e^{-\psi(y)}f(y)\ dy.
\end{equation}
Denote
\begin{equation}\label{def It}
I(t):=\norm{P_{s \to t}^{\psi}f}_{L_2(\Omega)}^2=\int_{\Omega }e^{2\psi(x)} \abs{u(t,x)}^2\ dx,\quad t\geq s.
\end{equation}
By using the equation \eqref{prob Diri} and the condition (H), we have
\begin{align}		\label{est I'}
I'(t) &=-2\int_{\Omega} (a^{ij}D_j  u+b^i u)D_i (e^{2\psi}u)+c^i D_i u e^{2\psi}u+du e^{2\psi}u \nonumber\\
&=-2\int_{\Omega} a^{ij}D_j  u (e^{2\psi} D_i u + 2 e^{2\psi}  u D_i \psi) +(c^i  -b^i)D_i u e^{2\psi}u +b^i  D_i( e^{2\psi} u^2)+d e^{2\psi}u^2  \nonumber\\
&\le -2\int_{\Omega} a^{ij}D_j  u (e^{2\psi} D_i u + 2 e^{2\psi}  u D_i \psi) +(c^i  -b^i)D_i u e^{2\psi}u  \nonumber\\
&= -2\int_{\Omega} a^{ij}D_j  u (e^{2\psi} D_i u + 2 e^{2\psi}  u D_i \psi) + \int_{\Omega} (b^i  -c^i)(D_i (e^{2\psi}u^2)-2u^2 e^{2\psi} D_i\psi) \nonumber\\
&\le  -2\nu \int_{\Omega }e^{2\psi}\abs{Du}^2+\frac{4 \gamma_1}{\nu} \int_{\Omega}e^{2\psi}\abs{u} \,\abs{Du}-2\int_{\Omega}(b^i-c^i) e^{2\psi}u^2D_i \psi.
\end{align}
By setting $u=0$ outside $\Omega$ and applying \eqref{eq_bmo_x}, we have
\begin{align}\label{b-c exp}
-2\int_{\Omega}(b^i-c^i) e^{2\psi}u^2 D_i \psi
&=2\int_{\Omega}\Phi^{ij}D_j  (e^{2\psi}u^2D_i \psi)\nonumber\\
&\le  C_n \norm{\Phi^{ij}}_{\mathrm{BMO}(\bR^n)}\,\norm{D_j (e^{2\psi}u^2 D_i \psi)}_{\mathscr{H}^{1}(\bR^n)},
\end{align}
By taking $f=e^\psi u$ and $g=e^\psi u D_i \psi$ in \eqref{ccpi}, we have 
\begin{align*}
\norm{D_j  (e^{2\psi}u^2D_i \psi)}_{\mathscr{H}^1}
&\le C_n \left\{ \left(\norm{e^{\psi}uD\psi}_{L_2}+\norm{e^{\psi}Du}_{L_2}\right) \norm{e^{\psi}u D\psi}_{L_2} \right.\\
&\qquad + \norm{e^{\psi}u}_{L_2} \left. \left( \norm{e^{\psi}u \abs{D\psi}^2}_{L_2}+\norm{e^{\psi}uD^2\psi}_{L_2}+\norm{e^{\psi}DuD\psi}_{L_2}\right)\right\}\\
&\le C_n \left((2\gamma_1^2+\gamma_2)\, \norm{e^\psi u}_{L_2}^2+2\gamma_1 \norm{e^\psi Du}_{L_2} \norm{e^\psi u}_{L_2}\right).
\end{align*}
Substituting the above into \eqref{b-c exp} and using \eqref{eq_bmo_x}, we obtain
\begin{multline*}
-2\int_{\Omega}(b^i-c^i) e^{2\psi}u^2 D_i \psi\\
\le C_n \Theta\left((2\gamma_1^2+\gamma_2)\int_{\Omega}e^{2\psi}u^2+2\gamma_1\left(\int_{\Omega}e^{2\psi}\abs{Du}^2\right)^{\frac12}\left(\int_{\Omega}e^{2\psi}u^2\right)^{\frac12}\right).
\end{multline*}
Coming back to \eqref{est I'} and using Young's inequality and H\"{o}lder's inequality, we obtain the differential inequality
\begin{equation}			\label{eq1523sat}
I'(t)\leq \left((4/\nu^3+C_n^2\Theta^2/\nu+2C_n\Theta)\gamma_1^2+C_n\Theta\gamma_2 \right) I(t).
\end{equation}
Recalling the initial condition $I(s)=\norm{f}_{L_2(\Omega)}^2$ and the definition \eqref{def It}, we obtain the $L_2 \to L_2$ estimate
\begin{align}\label{Ps-t f}
\norm{P_{s \to t}^\psi f}_{L_2(\Omega )}\leq e^{(\lambda \gamma_1^2+\mu\gamma_2)(t-s)}\norm{f}_{L_2(\Omega )},\quad \forall t \ge s,
\end{align}
where we set
\[
2\lambda:=4/\nu^3+C_n^2\Theta^2/\nu+2C_n\Theta\quad\text{and}\quad 2\mu:=C_n\Theta.
\]
With \eqref{Ps-t f} and Lemma~\ref{loc_bdd} at hand, we replicate the same arguments in \cite[p.~3028]{DK18} to obtain the $L_2 \to L_\infty$ estimate
\begin{equation}	\label{Pf Linfty}
\norm{P_{s \to t}^\psi f}_{L_{\infty}(\Omega)}\leq C(t-s)^{-\frac{n}{4}}e^{\gamma_1\sqrt{t-s}+(\lambda\gamma_1^2+\mu\gamma_2)(t-s)}\norm{f}_{L_2(\Omega)},\quad \forall t\ge s.
\end{equation}

Let the operator $Q_{t \to s}^\psi$ on $L_2(\Omega)$ for $s<t$ be given by
\[
Q_{t \to s}^\psi g(y)=e^{-\psi(y)}v(s,y)
\]
and denote
\[
J(s):=\norm{Q_{t \to s}^{\psi}f}_{L_2(\Omega)}^2=\int_{\Omega }e^{-2\psi} \abs{v(s,y)}^2\ dy,\quad s\leq t,
\]
where $v\in\mathring{V}_2^{1,0}((-\infty,t)\times \Omega)$ is the weak solution of the backward problem
\begin{equation}\label{Diri backpara prb gaussian}
\begin{cases}
P^{*}v=0,\\
v(t,\cdot)=e^{\psi}g.
\end{cases}
\end{equation}
Then similar to \eqref{est I'}, we have
\begin{align}		\label{est J'}
J'(s) &=2\int_{\Omega} (a^{ji}D_j  v+c^i v)D_i (e^{-2\psi}v)+b^i D_i ve^{-2\psi}v+e^{-2\psi}d v^2	\nonumber\\
&=2\int_{\Omega} a^{ji }D_j v D_i(e^{-2\psi}v)+b^i D_i(e^{-2\psi}v^2)+d e^{-2\psi}v^2+(c^i-b^i)v D_i(e^{-2\psi} v)	\nonumber\\
&\ge 2\int_{\Omega} a^{ji}D_j v(e^{-2\psi} D_i v -2e^{-2\psi} v D_i\psi )+(c^i-b^i) v D_i(e^{-2\psi}v)	\nonumber\\
&= 2\int_{\Omega} a^{ji}D_j v(e^{-2\psi} D_iv -2 e^{-2\psi} v D_i\psi)-2\int_{\Omega}(b^i-c^i) (D_i(e^{-2\psi} v^2)-e^{-2\psi} v D_iv)	\nonumber\\
&\ge 2\nu \int_{\Omega} e^{-2\psi} \abs{Dv}^2- \frac{4\gamma_1}{\nu} \int_{\Omega}e^{-2\psi}\abs{v} \,\abs{Dv}+2\int_{\Omega} (b^i-c^i) e^{-2\psi} v D_iv.	
\end{align}
Therefore, similar to \eqref{eq1523sat}, we have
\[
J'(s) \ge -\left((4/\nu^3+C_n^2\Theta^2/\nu+2C_n\Theta)\gamma_1^2+C_n\Theta\gamma_2 \right) J(s),
\]
and thus, similar to \eqref{Pf Linfty}, we obtain
\begin{equation}	\label{Pg Linfty}
\norm{Q_{t\rightarrow s}^{\psi}g}_{L_{\infty}(\Omega )}\leq C(t-s)^{-\frac{n}{4}}e^{\gamma_1\sqrt{t-s}+(\lambda \gamma_1^2+\mu\gamma_2)(t-s)}\norm{g}_{L_2(\Omega )},\quad \forall s\le t.
\end{equation}
From \eqref{prob Diri}, \eqref{Diri backpara prb gaussian}, and the definitions of $P_{s \to t}^\psi f$ and $Q_{t \to s}^\psi g$, we obtain the duality relation
\begin{equation}			\label{duality}
\int_{\Omega}(P_{s \to t}^\psi f) \, g=\int_{\Omega}f \,(Q_{t \to s}^\psi g).
\end{equation}
This combined with \eqref{Pg Linfty} yields the $L_1 \to L_2$ estimate
\begin{equation}		\label{Pf L2}
\norm{P_{s\rightarrow t}^{\psi}f}_{L_2(\Omega)}\leq C(t-s)^{-\frac{n}{4}}e^{\gamma_1\sqrt{t-s}+(\lambda \gamma_1^2+\mu\gamma_2)(t-s)}\norm{f}_{L_1(\Omega)},\quad \forall f\in C_c^\infty(\Omega).
\end{equation}
Then by noting $P_{s\to t}^{\psi}f=P_{(t+s)/2\to t}^{\psi}\left(P_{s \to (t+s)/2}^{\psi}\,f\right)$, we find from \eqref{Pf Linfty} and \eqref{Pf L2} that
\[
\norm{P_{s\rightarrow t}^{\psi}f}_{L_{\infty}(\Omega)}\leq C(t-s)^{-\frac{n}{2}}e^{\gamma_1\sqrt{2(t-s)}+(\lambda \gamma_1^2+\mu\gamma_2)(t-s)} \norm{f}_{L_1(\Omega)},\quad \forall f\in C_c^\infty(\Omega).
\]
For fixed $x$, $y\in\Omega$ with $x\neq y$, we obtain from the above estimate and \eqref{formula Pf} that
\begin{equation}		\label{eq2255sat}
e^{\psi(x)-\psi(y)}\abs{G(t,x,s,y)}\leq C(t-s)^{-\frac{n}{2}}e^{\gamma_1\sqrt{2(t-s)}+(\lambda \gamma_1^2+\mu\gamma_2)(t-s)}.
\end{equation}
This corresponds to \cite[(3.19)]{DK18} and by choosing an appropriate $\psi$, we obtain the Gaussian bound \eqref{Gaussian bound}.
See \cite[p. 3028]{DK18} for details.

\subsubsection{Case when $p>n$}
In the case when $p>n$, we combine the argument of Aronson \cite{Aronson68} with techniques in \cite{CDK08, HK04}.
Let $I(t)$ be as in \eqref{def It}.
It follows from \eqref{est I'} that
\[
I'(t) \le -2\nu \int_{\Omega} e^{2\psi}\abs{Du}^2+\frac{4 \gamma_1}{\nu} \int_{\Omega}e^{2\psi}\abs{u} \,\abs{Du}+2 \gamma_1 \int_{\Omega} \abs{\vec b-\vec c} \,e^{2\psi}\abs{u}^2.
\]
Let $\delta>0$ to be a number to be fixed later.
By integrating the above inequality in $t$ over $[t_1, t_2]$, where $s\le t_1 \le t_2 \le t_1+\delta$, and denoting
\[
S=[t_1, t_2] \times \Omega,
\]
we have
\begin{equation}		\label{eq1656sat}
I(t_2) + 2\nu \int_{S} e^{2\psi}\abs{Du}^2 \le I(t_1)+ \frac{4 \gamma_1}{\nu} \int_{S} e^{2\psi} \abs{u}\,\abs{Du}+ 2 \gamma_1\int_{S} \abs{\vec b-\vec c} \,e^{2\psi} \abs{u}^2.
\end{equation}
By Young's inequality, we have
\begin{equation}			\label{eq1702sat}
\frac{4 \gamma_1}{\nu} \int_{S} e^{2\psi} \abs{u}\,\abs{Du} 
\le \nu \int_S e^{2\psi} \abs{Du}^2 + \frac{4\gamma_1^2 \delta}{\nu^3}\,\norm{e^\psi u}_{L_{2,\infty}(S)}^2.
\end{equation}
Also, by H\"older's inequality, the condition (H), the embedding \eqref{embd_ineq},  and Young's inequality, we estimate
\begin{align*}
2\gamma_1 \int_{S} \abs{\vec b-\vec c} \,e^{2\psi} \abs{u}^2 &\le 2\gamma_1 \norm{\vec b-\vec c}_{L_{p,q}(S)}\norm{e^\psi u}_{L_2(S)} \norm{e^\psi u}_{L_{2p/(p-2),2q/(q-2)}(S)}\\
&\le 2 \gamma_1 \Theta \sqrt{\delta} \,\norm{e^\psi u}_{L_{2,\infty}(S)}\, \oldnorm{e^\psi u}_S
\le \frac{\nu}{4}\oldnorm{e^\psi u}_S^2+ \frac{4\gamma_1^2 \Theta^2 \delta}{\nu}\norm{e^\psi u}_{L_{2,\infty}(S)}^2,
\end{align*}
where we use the fact that $(\tilde p, \tilde q)=(\frac{2p}{p-2},\frac{2q}{q-2})$ satisfy \eqref{cond q r embed}.
Note that 
\begin{align*}
\oldnorm{e^\psi u}_S^2 &\le 2\norm{e^\psi u}_{L_{2,\infty}(S)}^2+2\norm{D(e^\psi u)}_{L_2(S)}^2\\
&\le 2\norm{e^\psi u}_{L_{2,\infty}(S)}^2+ 4 \gamma_1^2 \delta \norm{e^\psi u}_{L_{2,\infty}(S)}^2 + 4 \norm{e^\psi Du}_{L_2(S)}^2.
\end{align*}
Combining the above inequalities, we have
\begin{equation}			\label{eq1657sat}
2 \gamma_1 \int_{S} \abs{\vec b-\vec c} \,e^{2\psi} \abs{u}^2 \le \nu \int_S e^{2\psi} \abs{Du}^2 + \left(\frac{\nu}{2}+ \nu \gamma_1^2 \delta+\frac{4\gamma_1^2 \Theta^2 \delta}{\nu} \right)\norm{e^\psi u}_{L_{2,\infty}(S)}^2.
\end{equation}
By substituting \eqref{eq1702sat} and \eqref{eq1657sat} back to \eqref{eq1656sat}, we obtain
\begin{equation}			\label{eq2103sat}
I(t_2) \le I(t_1)+ \left(\frac{\nu}{2}+ \nu \gamma_1^2 \delta+\frac{4\gamma_1^2 \Theta^2 \delta}{\nu}+\frac{4\gamma_1^2 \delta}{\nu^3}\right)\norm{e^\psi u}_{L_{2,\infty}(S)}^2.
\end{equation}
Recall that $\nu \in (0,1)$.
We choose
\begin{equation}			\label{eq2131sat}
\delta=\frac{(3-2\nu)\nu^3}{4(\nu^4+4\Theta^2\nu^2+4)\gamma_1^2}\quad\text{so that}\quad
\frac{\nu}{2}+ \nu \gamma_1^2 \delta+\frac{4\gamma_1^2 \Theta^2 \delta}{\nu}+\frac{4\gamma_1^2 \delta}{\nu^3} = \frac34.
\end{equation}
Then, we take the supremum over $t_2 \in [t_1, t_1+\delta]$ in \eqref{eq2103sat} to get
\[
\max_{t_1 \le t \le  t_1+\delta} I(t) \le 4  I(t_1).
\]
In particular, by take $t_1=s$ and iterating, we have
\[
I(t) \le 4^j I(s)=4^j \norm{f}_{L_2(\Omega)}^2  \quad\text{if } s+(j-1)\delta \le t \le s+ j \delta,
\]
which combined with \eqref{eq2131sat} yields
\begin{equation}				\label{eq1741th}
I(t) \le 4 e^{2\mu \gamma_1^2 (t-s)} \norm{f}_{L_2(\Omega)}^2,\quad \forall t \ge s, \quad\text{where }\; \mu=\frac{2 (\nu^4+4\Theta^2\nu^2+4)\ln 4}{(3-2\nu)\nu^3}.
\end{equation}
which is equivalent to
\begin{equation}				\label{eq1755th}
\norm{P_{s \to t}^\psi f}_{L_2(\Omega)}\le  2 e^{\mu \gamma_1^2 (t-s)} \norm{f}_{L_2(\Omega)},\quad \forall t\ge s.
\end{equation}
With the $L_2 \to L_2$ estimate \eqref{eq1755th} and Lemma~\ref{loc_bdd} at hand, we replicate the same argument in \cite[p.~3028]{DK18} to obtain the $L_2 \to L_\infty$ estimate (c.f. \eqref{Pf Linfty})
\[
\norm{P_{s \to t}^\psi f}_{L_{\infty}(\Omega)}\le C(t-s)^{-\frac{n}{4}}e^{\gamma_1\sqrt{t-s}+ \mu \gamma_1^2 (t-s)}\norm{f}_{L_2(\Omega)},\quad \forall t\ge s.
\]
Similarly, we obtain from \eqref{est J'} that
\[
J(s) \le 4 e^{2\mu \gamma_1^2 (t-s)} \norm{g}_{L_2(\Omega)}^2,\quad \forall s\leq t.
\]
which, combined with Lemma~\ref{loc_bdd} and duality relation \eqref{duality}, yields the $L_1 \to L_2$ estimate (c.f. \eqref{Pf L2})
\[
\norm{P_{s \to t}^\psi f}_{L_2(\Omega)}\le  C(t-s)^{-\frac{n}{4}}e^{\gamma_1\sqrt{t-s}+ \mu \gamma_1^2 (t-s)} \norm{f}_{L_1(\Omega)},\quad  t\ge s,\quad \forall f\in C_c^\infty(\Omega).
\]
Then, similar to \eqref{eq2255sat}, for $x\neq y$, we have
\[
e^{\psi(x)-\psi(y)}\abs{G(t,x,s,y)}\leq C(t-s)^{-\frac{n}{2}}e^{\gamma_1\sqrt{2(t-s)}+\mu \gamma_1^2 (t-s)}, \quad t>s.
\]
This corresponds to \cite[(5.8)]{CDK08}.
Note that $\mu$, which is specified in \eqref{eq1741th}, depends only on $\nu$ and $\Theta$.
By choosing the function $\psi$ appropriately, we obtain the Gaussian bound \eqref{Gaussian bound}. See \cite[p. 1670]{CDK08} for details.

\section{Proof of Lemma \ref{loc_bdd}}			\label{sec7}
The proof is based on an original idea of De Giorgi \cite{DeGiorgi} in the parabolic context as appears in \cite{LSU}.
See  Seregin et al. \cite{SSSZ} and Nazarov and Ural'tseva \cite{NU12} for related results.
We restrict ourselves to the case when $u$ is a weak solution of $Pu=f$ in $Q_r^{-}$. The proof for the other case requires just a routine adjustment and we leave the details to the readers.

\subsection{Case when $p>n$}
We shall first treat the case when $p>n$ so that $q<\infty$.
Let us denote
\[
v =(u-k)_{+}=\max(u-k, 0),
\]
where $k> 0$ is to be chosen, and let $\zeta:\bR^{n+1} \to [0,1]$ be a smooth cut-off function such that
\[
\supp (\zeta) \cap \set{t \le t_0} \subset (t_0-r^2, t_0] \times B_r(x_0).
\]
In what follows we shall write $\Omega_r=\Omega\cap B_r(x_0)$ and $Q_r^{-}=Q_r^{-}(X_0)$ for brevity.
By testing $Pu=f$ with $\zeta^2v$, using the assumption that $d-\Div\vec b\ge 0$ together with $\zeta^2 uv \ge 0$, and noting that $Du = Dv$ on the set $\set{u>k}=\set{v> 0}$, we obtain
\begin{multline*}
\frac12 \int_{\Omega_r}  \zeta^2 v^2(t_1,x)\,dx+\int_{t_0-r^2}^{t_1} \!\int_{\Omega_r} \zeta^2 a^{ij}D_j v D_iv \,dxdt +\int_{t_0-r^2}^{t_1} \!\int_{\Omega_r} 2 a^{ij} \zeta D_j v D_i \zeta v \,dxdt\\
\le \int_{t_0-r^2}^{t_1} \!\int_{\Omega_r} \zeta\partial_t \zeta v^2\,dxdt+\int_{t_0-r^2}^{t_1} \!\int_{\Omega_r}\zeta^2 v (b^i-c^i) D_iv\,dxdt+\int_{t_0-r^2}^{t_1} \!\int_{\Omega_r} f \zeta^2 v\,dxdt
\end{multline*}
for all $t_1$ satisfying $t_0-r^2 \le t_1 \le t_0$.
By the assumption that $\Div (\vec b-\vec c) \ge 0$, we have
\begin{equation*}			
\int_{t_0-r^2}^{t_1} \!\int_{\Omega_r}\zeta^2 v (b^i-c^i) D_i v\le -\int_{t_0-r^2}^{t_1} \!\int_{\Omega_r} \zeta (b^i-c^i) v^2 D_i\zeta.
\end{equation*}
Then by the condition (H1) and Young's inequality, we obtain
\begin{multline}			\label{eq1455th}
\frac12 \norm{\zeta v}_{L_{2,\infty}(Q^{-}_r)}^2+ \frac{\nu}{2} \int_{Q^{-}_r} \zeta^2 \abs{Dv}^2 \\
\le 2 \int_{Q_r^{-}} \abs{\partial_t \zeta} v^2+ \frac{4}{\nu^3}\int_{Q^{-}_r} \abs{D\zeta}^2 v^2 - 2\int_{Q^{-}_r}\zeta (b^i-c^i) v^2 D_i \zeta +2\int_{Q^{-}_r} f\zeta^2 v.
\end{multline}
We estimate the last two terms as follows.
Note that $\zeta v \in \mathring{V}_2(Q^{-}_r)$.
By using H\"{o}lder's inequality, Young's inequality, and the embedding \eqref{embd_ineq}, we have
\begin{align}			\label{est f int}
\int_{Q^{-}_r} f\zeta^2 v &\le \norm{f}_{L_\infty(Q^{-}_r)} \norm{\zeta v}_{L_{2(n+2)/n}(Q^{-}_r)} \,\abs{Q^{-}_r \cap\set{u>k}}^{1-\frac{n}{2(n+2)}}		\nonumber\\
&\le \beta \norm{f}_{L_\infty(Q^{-}_r)} \oldnorm{\zeta v}_{Q^{-}_r} \,\abs{Q^{-}_r \cap\set{u>k}}^{\frac{n+4}{2(n+2)}}		\nonumber\\
&\leq \frac{\nu}{64} \oldnorm{\zeta v}_{Q^{-}_r}^2+ \frac{16\beta^2}{\nu} \norm{f}_{L_\infty(Q^{-}_r)}^2 \,
\abs{Q^{-}_r \cap \set{u>k}}^{\frac{n+4}{n+2}}.
\end{align}
By H\"{o}lder's inequality, the embedding \eqref{embd_ineq}, and Young's inequality, we obtain
\begin{align}		\label{est penul term}
-\int_{Q^{-}_r} \zeta (b^i-c^i)v^2 D_i \zeta &\le \norm{\vec b-\vec c}_{L_{p,q}(Q^{-}_r)}\norm{\zeta v}_{L_{2p/(p-2),2q/(q-2)}(Q^{-}_r)}\norm{D\zeta v}_{L_2(Q^{-}_r)} \nonumber\\
&\le \beta \Theta \oldnorm{\zeta v}_{Q^{-}_r}\norm{D\zeta v}_{L_2(Q^{-}_r)} \nonumber\\
&\le \frac{\nu}{64} \oldnorm{\zeta v}_{Q^{-}_r}^2+\frac{16\beta^2 \Theta^2}{\nu}  \norm{D\zeta}_{L_\infty}^2 \norm{v}_{L_2(Q^{-}_r)}^2.
\end{align}
where we used the fact that the pair $(\frac{2p}{p-2}, \frac{2q}{q-2})$ satisfy the condition \eqref{cond q r embed}.

It follows from \eqref{eq1455th}, \eqref{est f int}, and  \eqref{est penul term} that
\begin{multline*}		
\frac12 \norm{\zeta v}_{L_{2,\infty}(Q^{-}_r)}^2+ \frac{\nu}{2} \norm{\zeta Dv}_{L_2(Q^{-}_r)}^2
\le C(\nu,\beta, \Theta) \left(\norm{\partial_t \zeta}_{L_\infty} + \norm{D\zeta}_{L_\infty}^2 \right) \norm{v}_{L_2(Q^{-}_r)}^2 \\
+\frac{32\beta^2}{\nu} \norm{f}_{L_\infty(Q^{-}_r)}^2 \, \abs{Q^{-}_r \cap \set{u>k}}^{\frac{n+4}{n+2}} 
+\frac{\nu}{16} \oldnorm{\zeta v}_{Q^{-}_r}^2.
\end{multline*}
Since
\begin{align*}		
\oldnorm{\zeta v}_{Q^{-}_r}^2 &\le 2 \norm{\zeta v}_{L_{2,\infty}(Q^{-}_r)}^2+ 2\norm{D(\zeta v)}_{L_2(Q^{-}_r)}^2	\nonumber\\
&\le 4 \norm{\zeta v}_{L_{2,\infty}(Q^{-}_r)}^2+ 4\norm{D\zeta v}_{L_2(Q^{-}_r)}^2+ 4\norm{\zeta Dv}_{L_2(Q^{-}_r)}^2
\end{align*}
and $0<\nu <1$, it follows that
\begin{multline}		\label{eq1811th}
\oldnorm{\zeta (u-k)_{+}}_{Q^{-}_r}^2\le C \left(\norm{\partial_t \zeta}_{L_\infty} + \norm{D\zeta}_{L_\infty}^2 \right) \norm{(u-k)_{+}}_{L_2(Q^{-}_r)}^2\\
+C\norm{f}_{L_\infty(Q^{-}_r)}^2 \,\abs{Q^{-}_r \cap \set{u>k}}^{\frac{n+4}{n+2}},
\end{multline}
where $C=C(\nu, \beta, \Theta)=C(n, \nu, p, \Theta)$.
On the other hand, by H\"older's inequality we have
\begin{equation}		\label{eq1320th}
\norm{\zeta (u-k)_{+}}_{L_2(Q^{-}_r)}^2 \le \abs{Q^{-}_r \cap \set{u>k}}^{\frac{2}{n+2}}\,\norm{\zeta (u-k)_{+}}_{L_{2(n+2)/n}(Q^{-}_r)}^2.
\end{equation}
It follows from \eqref{eq1320th}, \eqref{eq1811th}, and the embedding \eqref{embd_ineq} that
\begin{multline}			\label{eq1558w}
\norm{\zeta (u-k)_{+}}_{L_2(Q^{-}_r)}^2 \le C \left(\norm{\partial_t \zeta}_{L_\infty} + \norm{D\zeta}_{L_\infty}^2 \right) \norm{(u-k)_{+}}_{L_2(Q^{-}_r)}^2\,\abs{Q^{-}_r \cap \set{u>k}}^{\frac{2}{n+2}} \\
+Cr^2 \norm{f}_{L_\infty(Q^{-}_r)}^2\,\abs{Q^{-}_r \cap \set{u>k}}^{\frac{n+4}{n+2}}.
\end{multline}

Now, for $m=1$, $2$, $\ldots$, we set
\[
r_m =r\left(\frac{1}{2}+\frac{1}{2^{m}}\right),\quad k_m=k\left(1-\frac{1}{2^{m-1}}\right),\quad Q_m :=Q_{r_m}^{-}=(t_0-r_m ^2,t_0]\times (B_{r_m }(x_0) \cap \Omega),
\]
and let $\zeta_m:\bR^{n+1} \to [0,1]$ be smooth cut-off functions such that  $\zeta_m=1$ on $Q_{m+1}$, $\supp (\zeta_m) \cap \set{t \le t_0} \subset Q_{m}$, and 
\begin{equation*}				
\abs{\partial_t\zeta_m}+\abs{D\zeta_m}^2+ \abs{D^2\zeta_m} \le  \frac{100\cdot 2^{2m}}{r^2}.
\end{equation*}
By taking $\zeta=\zeta_m$, $r=r_m$, and $k=k_{m+1}$ in \eqref{eq1558w}, and then using obvious inequalities
\begin{equation}		\label{eq15.42sat}
(k_{m+1}-k_m)^2\, \abs{Q_m \cap \set{u>k_{m+1}}}=\int_{Q_m \cap \set{u>k_{m+1}}} \abs{u-k_m}^2  \le \int_{Q_m} (u-k_m)_{+}^2
\end{equation}
and
\begin{equation*}		
\norm{(u-k_{m+1})_{+}}_{L_2(Q_m)} \le \norm{(u-k_m)_{+}}_{L_2(Q_m)},
\end{equation*}
we have
\begin{multline}		\label{eq1452th}
\norm{(u-k_{m+1})_{+}}_{L_2(Q_{m+1})}^2\le \frac{C2^{2m}}{r^2}  \left(\frac{2^{2m}}{k^2}\right)^{\frac{2}{n+2}}\,\norm{(u-k_m)_{+}}_{L_2(Q_m)}^{\frac{2(n+4)}{n+2}}\\
+Cr^2 \norm{f}_{L_\infty(Q^{-}_r)}^2 \left(\frac{2^{2m}}{k^2}\right)^{\frac{n+4}{n+2}}\,\norm{(u-k_m)_{+}}_{L_2(Q_m)}^{\frac{2(n+4)}{n+2}}.
\end{multline}
Let us denote
\[
Y_m := \frac{1}{k^2 r^{n+2}} \norm{(u-k_m )_{+}}_{L_2(Q_m )}^2
\]
and assume 
\[
k \ge r^2 \norm{f}_{L_\infty(Q^{-}_r)}.
\]
Then, it follows from \eqref{eq1452th} that
\begin{align*}
Y_{m+1} &\le C 2^{\frac{2m(n+4)}{n+2}}Y_m^{\frac{n+4}{n+2}} +C 2^{\frac{2m(n+4)}{n+2}} k^{-2} r^4 \norm{f}_{L_\infty(Q^{-}_r)}^2 Y_m^{\frac{n+4}{n+2}} \\
& \le K2^{2m\frac{n+4}{n+2}} Y_m^{\frac{n+4}{n+2}},
\end{align*}
where $K=K(n,\nu, p, \Theta)>0$.
By a well-known lemma on fast geometric convergence (see, e.g., \cite[Lemma 15.1]{DiBenedetto}), it follows that $Y_m \to 0$ provided
\[
Y_1=\frac{1}{k^2 r^{n+2}} \norm{u_{+}}_{L_2(Q^{-}_r )}^2 \le \delta^2
\]
for some $\delta=\delta(n, K)=\delta(n, \nu, p, \Theta)>0$.
Therefore, by taking
\[
k= \max\left(r^2 \norm{f}_{L_\infty(Q^{-}_r)}, \delta^{-1} r^{-\frac{n+2}{2}} \norm{u_{+}}_{L_2(Q^{-}_r)}\right),
\]
we see that
\[
u\le k\quad\text{in }\;Q^{-}_{r/2}.
\]
By applying the same estimate to $-u$, we obtain \eqref{loc_bdd_est}.
\qed

\subsection{Case when $p=n$}
We now treat the case when $p=n$ and $q=\infty$.
We proceed the same as in the case when $p>n$ until we reach \eqref{est penul term}, where by using \eqref{eq_bmo_x}, we instead obtain
\begin{align}		\label{est hardy}
-\int_{Q^{-}_r} \zeta (b^i-c^i) v^2 D_i\zeta &=\int_{Q^{-}_r} \left(\Phi^{ij}- \overline\Phi{}^{ij}(t)\right)D_j  (\zeta v^2 D_i \zeta)	\nonumber\\
&=\int_{Q^{-}_r} \left(\Phi^{ij}- \overline\Phi{}^{ij}(t)\right) (v^2 D_j(\zeta D_i \zeta)+2v D_j  v \zeta D_i \zeta),
\end{align}
where we set
\[
\overline\Phi{}^{ij}(t) =\frac{1}{\abs{B_r}} \int_{B_r}\Phi^{ij}(t,x)\,dx.
\]
Fix a number $s\in (2,\frac{2(n+2)}{n})$.
By using H\"{o}lder's inequality and the John-Nirenberg inequality, we estimate
\begin{align}\label{first est}
\int_{Q^{-}_r} &\left(\Phi^{ij}- \overline\Phi{}^{ij}(t)\right) v^2 D_j(\zeta D_i \zeta)dxdt\nonumber\\
&\le \norm{D_j(\zeta D_i\zeta)}_{L_\infty} \int_{Q^{-}_r} \abs{\Phi^{ij}- \overline\Phi{}^{ij}(t)} v^2 \,dxdt\nonumber\\
&\le \norm{D_j(\zeta D_i\zeta)}_{L_\infty} \left(\int_{t_0-r^2}^{t_0}\!\int_{B_r} \abs{\Phi^{ij}(t,x)- \overline\Phi{}^{ij}(t)}^{\frac{s}{s-2}}\,dx dt\right)^{\frac{s-2}{s}} \left(\int_{Q^{-}_r} \abs{v}^s\,dxdt\right)^{\frac{2}{s}}\nonumber\\
&\le \norm{D_j(\zeta D_i\zeta)}_{L_\infty} (r^2 \abs{B_r})^{\frac{s-2}{s}} \left(\fint_{t_0-r^2}^{t_0}\!\fint_{B_r} \abs{\Phi^{ij}(t,x)- \overline\Phi{}^{ij}(t)}^{\frac{s}{s-2}}\,dx dt\right)^{\frac{s-2}{s}}\norm{v}_{L_s(Q^{-}_r)}^2 \nonumber\\
&\le C\left(\norm{D\zeta}_{L_\infty}^2 +\norm{D^2 \zeta}_{L_\infty} \right) r^{\frac{(n+2)(s-2)}{s}} \Theta\norm{v}_{L_s(Q^{-}_r)}^2,
\end{align}
where $C=C(n,s)$.
Similarly, we estimate
\begin{align}\label{second est}
\int_{Q^{-}_r} &\left(\Phi^{ij}- \overline\Phi{}^{ij}(t)\right) v D_j v \zeta D_i \zeta\,dxdt \nonumber\\
&\le \left(\int_{t_0-r^2}^{t_0}\!\int_{B_r} \abs{\Phi^{ij}(t,x)- \overline\Phi{}^{ij}(t)}^{\frac{2s}{s-2}}\ dx\ dt\right)^{\frac{s-2}{2s}} \left(\int_{Q^{-}_r} \abs{\zeta Dv}^2\right)^{\frac{1}{2}} \left(\int_{Q^{-}_r} \abs{v D\zeta}^s\right)^{\frac{1}{s}} \nonumber\\
&\le C r^{\frac{(n+2)(s-2)}{2s}}\, \Theta \norm{\zeta Dv}_{L_2(Q^{-}_r)} \norm{D\zeta}_{L_\infty}  \norm{v}_{L_s(Q^{-}_r)} \nonumber\\
&\leq \frac{\nu}{8}\, \norm{\zeta Dv}_{L_2(Q^{-}_r)}^2+\frac{2C^2 \Theta^2}{\nu} r^{\frac{(n+2)(s-2)}{s}}\,\norm{D\zeta}_{L_\infty}^2  \norm{v}_{L_s(Q^{-}_r)}^2.
\end{align}
where we used Young's inequality at the last step.
	
Coming back to \eqref{est hardy} and using \eqref{first est} and \eqref{second est}, we obtain
\begin{multline}		\label{bound b-c}
-\int_{Q^{-}_r} \zeta (b^i-c^i) v^2 D_i \zeta \le  \frac{\nu}{8} \int_{Q^{-}_r} \zeta^2 \abs{Dv}^2 \\
+C(n,s, \Theta,\nu) \left(\norm{D\zeta}_{L_\infty}^2+\norm{D^2 \zeta}_{L_\infty}\right) r^{\frac{(n+2)(s-2)}{s}} \norm{v}_{L_s(Q^{-}_r)}^2.
\end{multline}
Using \eqref{bound b-c} instead of \eqref{est penul term}, we obtain
similar to \eqref{eq1811th} that 
\begin{multline*}		
\oldnorm{\zeta (u-k)_{+}}_{Q^{-}_r}^2\le C \left(\norm{\partial_t \zeta}_{L_\infty} + \norm{D\zeta}_{L_\infty}^2 +\norm{D^2\zeta}_{L_\infty} \right)r^{\frac{(n+2)(s-2)}{s}} \norm{(u-k)_{+}}_{L_s(Q^{-}_r)}^2\\
+Cr^{n+4-\frac{2(n+2)}{s}} \norm{f}_{L_\infty(Q^{-}_r)}^2 \,\abs{Q^{-}_r \cap \set{u>k}}^{\frac{2}{s}}.
\end{multline*}
Also, similar to \eqref{eq1320th}, we have
\begin{equation*}		
\norm{\zeta (u-k)_{+}}_{L_s(Q^{-}_r)}^2 \le \abs{Q^{-}_r \cap \set{u>k}}^{\frac{2}{s}-\frac{n}{n+2}}\,\norm{\zeta (u-k)_{+}}_{L_{2(n+2)/n}(Q^{-}_r)}^2.
\end{equation*}
Take $r_m$, $k_m$, $Q_m$, and $\zeta_m$ as before.
By setting $\zeta=\zeta_m$, $r=r_m$, and $k=k_{m+1}$ in the preceding two inequalities, applying the embedding \eqref{embd_ineq}, and using
\begin{equation*}		
(k_{m+1}-k_m)^s\, \abs{Q_m \cap \set{u>k_{m+1}}} \le \int_{Q_m} (u-k_m)_{+}^s,
\end{equation*}
instead of \eqref{eq15.42sat}, we obtain
\begin{multline}		\label{eq1221th}
\norm{(u-k_{m+1})_{+}}_{L_s(Q_{m+1})}^2\le \frac{C2^{2m}}{r^2} r^{\frac{(n+2)(s-2)}{s}}\left(\frac{2^{sm}}{k^s}\right)^{\frac{2}{s}-\frac{n}{n+2}}\,\norm{(u-k_m)_{+}}_{L_s(Q_m)}^{4-\frac{sn}{n+2}}\\
+C r^{n+4-\frac{2(n+2)}{s}}  \norm{f}_{L_\infty(Q^{-}_r)}^2 \left(\frac{2^{sm}}{k^s}\right)^{\frac{4}{s}-\frac{n}{n+2}}\,\norm{(u-k_m)_{+}}_{L_s(Q_m)}^{4-\frac{sn}{n+2}},
\end{multline}
which corresponds to \eqref{eq1452th}.
Now, if we set
\[
Y_m := \frac{1}{k^2 r^{2(n+2)/s}} \norm{(u-k_m )_{+}}_{L_s(Q_m )}^2,
\]
then it follows from \eqref{eq1221th} that
\begin{align*}
Y_{m+1} &\le C 2^{4m-\frac{sn}{n+2}m} Y_m^{2-\frac{sn}{2(n+2)}}+ C r^4 \norm{f}_{L_\infty(Q^{-}_r)}^2 2^{2m+\frac{4s}{n+2}m} k^{-2}Y_m^{2-\frac{sn}{2(n+2)}}\\
&\le C 2^{2m+\frac{4s}{n+2}m} Y_m^{2-\frac{sn}{2(n+2)}}
\end{align*}
provided that $k \ge r^2 \norm{f}_{L_\infty(Q^{-}_r)}$.
By the same argument involving fast geometric convergence as above, we see that
\[
u\le C \left(r^2\norm{f}_{L_\infty(Q^{-}_r)}+r^{-\frac{n+2}{s}} \norm{u_+}_{L_s(Q^{-}_r)}\right)\quad\text{in }\;Q_{r/2}^{-}.
\]
By applying the same estimate to $-u$, and applying a well-known covering argument (see, e.g., \cite[pp.~80--82]{Giaquinta93}), we can replace the number $s$ by $2$ and get the estimate \eqref{loc_bdd_est}.
\qed



\begin{thebibliography}{99}

\bibitem{Aronson67}
Aronson, D. G.
\textit{Bounds for the fundamental solution of a parabolic equation}.
Bull. Amer. Math. Soc. \textbf{73} (1967), 890--896.

\bibitem{Aronson68}
Aronson, D. G.
\textit{Non-negative solutions of linear parabolic equations}.
Ann. Scuola Norm. Sup. Pisa Cl. Sci. (3) \textbf{22} (1968), 607--694.

\bibitem{CDK08}
Cho, Sungwon; Dong, Hongjie; Kim, Seick.
\textit{On the Green's matrices of strongly parabolic systems of second order}.
Indiana Univ. Math. J. \textbf{57} (2008), no. 4, 1633--1677.

\bibitem{CLMS}
Coifman, R.; Lions, P.-L.; Meyer, Y.; Semmes, S.
\textit{Compensated compactness and Hardy spaces}.
J. Math. Pures Appl. (9) \textbf{72} (1993), no. 3, 247--286.

\bibitem{Davies}
Davies, E. B.
\textit{Explicit constants for Gaussian upper bounds on heat kernels}.
Amer. J. Math. \textbf{109} (1987), no. 2, 319--333.

\bibitem{DeGiorgi}
De Giorgi, Ennio.
\textit{Sulla differenziabilit\`{a} e l'analiticit\`{a} delle estremali degli integrali multipli regolari}. (Italian)
Mem. Accad. Sci. Torino. Cl. Sci. Fis. Mat. Nat. (3) \textbf{3} 1957, 25--43.


\bibitem{DiBenedetto}
DiBenedetto, Emmanuele.
\textit{Partial differential equations}. Second edition.
Cornerstones. Birkh\"auser Boston, Ltd., Boston, MA, 2010.

\bibitem{DK09}
Dong, Hongjie; Kim, Seick.
\textit{Green's matrices of second order elliptic systems with measurable coefficients in two dimensional domains}.
Trans. Amer. Math. Soc. \textbf{361} (2009), no. 6, 3303--3323.

\bibitem{DK18}
Dong, Hongjie; Kim, Seick.
\textit{Fundamental solutions for second-order parabolic systems with drift terms}.
Proc. Amer. Math. Soc. \textbf{146} (2018), no. 7, 3019--3029.

\bibitem{FS86}
Fabes, E. B.; Stroock, D. W.
\textit{A new proof of Moser's parabolic Harnack inequality using the old ideas of Nash}.
Arch. Rational Mech. Anal. \textbf{96} (1986), no. 4, 327--338.

\bibitem{Giaquinta93}
Giaquinta, Mariano.
\textit{Introduction to regularity theory for nonlinear elliptic systems}.
Lectures in Mathematics ETH Z\"urich. Birkh\"auser Verlag, Basel, 1993.

\bibitem{GW82}
Gr\"uter, Michael; Widman, Kjell-Ove.
\textit{The Green function for uniformly elliptic equations}.
Manuscripta Math. \textbf{37} (1982), no. 3, 303--342.

\bibitem{HK04}
Hofmann, Steve; Kim, Seick.
\textit{Gaussian estimates for fundamental solutions to certain parabolic systems}.
Publ. Mat. \textbf{48} (2004), no. 2, 481--496.

\bibitem{HK07}
Hofmann, Steve; Kim, Seick.
\textit{The Green function estimates for strongly elliptic systems of second order}.
Manuscripta Math. \textbf{124} (2007), no. 2, 139--172. 

\bibitem{KS19}
Kim, Seick; Sakellaris, Georgios.
\textit{Green's function for second order elliptic equations with singular lower order coefficients}.
Comm. Partial Differential Equations \textbf{44} (2019), no. 3, 228--270.

\bibitem{KT01}
Koch, Herbert; Tataru, Daniel.
\textit{Well-posedness for the Navier-Stokes equations}.
Adv. Math. \textbf{157} (2001), no. 1, 22--35.

\bibitem{LSU}
Lady\v{z}enskaja, O. A.; Solonnikov, V. A.; Ural'ceva, N. N.
\textit{Linear and quasilinear equations of parabolic type}. (Russian)
Translated from the Russian by S. Smith. Translations of Mathematical Monographs, Vol. 23 American Mathematical Society, Providence, R.I. 1968

\bibitem{LSW}
Littman, W.; Stampacchia, G.; Weinberger, H. F.
\textit{Regular points for elliptic equations with discontinuous coefficients}.
Ann. Scuola Norm. Sup. Pisa Cl. Sci. (3) \textbf{17} (1963), 43--77.

\bibitem{Moser}
Moser, J\"{u}rgen.
\textit{A Harnack inequality for parabolic differential equations}.
Comm. Pure Appl. Math. \textbf{17} (1964), 101--134.

\bibitem{Nash}
Nash, J.
\textit{Continuity of solutions of parabolic and elliptic equations}.
Amer. J. Math. \textbf{80} (1958), 931--954.


\bibitem{NU12}
Nazarov, A. I.; Ural'tseva, N. N.
\textit{The Harnack inequality and related properties of solutions of elliptic and parabolic equations with divergence-free lower-order coefficients}. (Russian. Russian summary)
Algebra i Analiz \textbf{23} (2011), no. 1, 136--168; translation in St. Petersburg Math. J. \textbf{23} (2012), no. 1, 93--115.

\bibitem{QX1}
Qian, Zhongmin; Xi, Guangyu.
\textit{Parabolic equations with singular divergence-free drift vector fields}.
J. Lond. Math. Soc. (2) \textbf{100} (2019), no. 1, 17--40. 
 
 
\bibitem{QX2}
Qian, Zhongmin; Xi, Guangyu.
\textit{Parabolic equations with divergence-free drift in space $L^{l}_{t}L^{q}_{x}$}.
Indiana Univ. Math. J. \textbf{68} (2019), no. 3, 761--797.

\bibitem{Semenov}
Semenov, Yu. A.
\textit{Regularity theorems for parabolic equations}.
J. Funct. Anal. \textbf{231} (2006), no. 2, 375--417.

\bibitem{SSSZ}
Seregin, Gregory; Silvestre, Luis; \v{S}ver\'{a}k, Vladim\'{i}r; Zlato\v{s}, Andrej.
\textit{On divergence-free drifts}.
J. Differential Equations \textbf{252} (2012), no. 1, 505--540.

\bibitem{Zhang}
Zhang, Qi S.
\textit{A strong regularity result for parabolic equations}.
Comm. Math. Phys. \textbf{244} (2004), no. 2, 245--260.

\end{thebibliography}
\end{document}